\documentclass[oneside]{amsart}

\usepackage{amsmath}
\usepackage{amsthm}
\usepackage{amsfonts}
\usepackage{dsfont}
\usepackage{enumitem}
\usepackage{tikz}
\usepackage{bbm}
\usepackage[margin=42mm]{geometry}

\newcommand{\reals}{\mathbb{R}}

\newcommand{\complex}{\mathbb{C}}

\newcommand{\para}[1]{\left(#1\right)}
\newcommand{\paraa}[1]{\big(#1\big)}

\newcommand{\spacearound}[1]{\quad#1\quad}

\renewcommand{\implies}{\spacearound{\Rightarrow}}

\newtheorem{theorem}{Theorem}[section]
\newtheorem{corollary}[theorem]{Corollary}

\newtheorem{proposition}[theorem]{Proposition}
\newtheorem{example}[theorem]{Example}
\theoremstyle{definition}
\newtheorem{definition}[theorem]{Definition}
\theoremstyle{remark}
\newtheorem{remark}[theorem]{Remark}
\numberwithin{equation}{section}

\renewcommand{\mid}{\mathds{1}}
\newcommand{\A}{\mathcal{A}}

\renewcommand{\d}{\partial}

\newcommand{\g}{\mathfrak{g}}

\newcommand{\Sthreetloc}{S^3_{\theta,\text{loc}}}

\newcommand{\C}{\mathcal{C}}

\newcommand{\nablah}{\hat{\nabla}}

\newcommand{\mathand}{\quad\text{and}\quad}

\newcommand{\psih}{\widehat{\psi}}
\newcommand{\psit}{\tilde{\psi}}
\newcommand{\MPsi}{M_\Psi}

\newcommand{\Hom}{(\phi,\psi,\ModHom)}
\newcommand{\Alg}{\mathcal{A}}
\newcommand{\Algprim}{\mathcal{A}'}
\newcommand{\Mod}{M}
\newcommand{\Modprim}{M'}
\newcommand{\Der}{\mathfrak{g}}
\newcommand{\Derprim}{\mathfrak{g}'}
\newcommand{\ModHom}{\psih}
\newcommand{\der}{\partial}
\newcommand{\Real}{\mathbb{R}}
\newcommand{\RCa}{(\Alg, \Der,\Mod,\varphi)}
\newcommand{\RCb}{(\Algprim, \Derprim,\Modprim,\varphi ')}

\newcommand{\ExtPsi}{\operatorname{Ext}_{\Psi}}

\newcommand{\Mani}{\Sigma}
\newcommand{\ManiSmooth}{\mathcal{C}^{\infty}(\Mani)}

\newcommand{\FofG}[2]{#1\left(#2\right)}

\newcommand{\DerSet}{\operatorname{Der}}
\newcommand{\Vect}{\operatorname{Vect}}
\newcommand{\CA}{C_{\A}}
\newcommand{\CAp}{C_{\A'}}

\newcommand{\Torus}{T^2_{\theta}}
\newcommand{\One}{\mathbbm{1}}
\newcommand{\Sphere}{S^3_{\theta}}

\newcommand{\hh}{\hat{h}}

\title[]{Noncommutative minimal embeddings and\\ morphisms of pseudo-Riemannian calculi}
\author{Joakim Arnlind and Axel Tiger Norkvist}

\address[Joakim Arnlind]{Dept. of Math.\\
Link\"oping University\\
581 83 Link\"oping\\
Sweden}
\email{joakim.arnlind@liu.se}

\address[Axel Tiger Norkvist]{Dept. of Math.\\
Link\"oping University\\
581 83 Link\"oping\\
Sweden}
\email{axel.tiger.norkvist@liu.se}



\begin{document}

\begin{abstract}
  In analogy with classical submanifold theory, we introduce morphisms
  of real metric calculi together with noncommutative embeddings. We
  show that basic concepts, such as the second fundamental form and
  the Weingarten map, translate into the noncommutative setting and,
  in particular, we prove a noncommutative analogue of Gauss'
  equations for the curvature of a submanifold. Moreover, the mean
  curvature of an embedding is readily introduced, giving a natural
  definition of a noncommutative minimal embedding, and we illustrate
  the novel concepts by considering the noncommutative torus as a
  minimal surface in the noncommutative 3-sphere.
\end{abstract}

\maketitle

\section{Introduction}

\noindent
In recent years, a lot of progress has been made in understanding the
Riemannian aspects of noncommutative geometry. The Levi-Civita
connection of a metric plays a crucial role in classical Riemannian
geometry and it is important to understand to what extent a
corresponding noncommutative theory exists. Several impressive results
exist, which compute the curvature of the noncommutative torus from
the heat kernel expansion and consider analogues of the classical
Gauss-Bonnet theorem
\cite{ct:gaussBonnet,fk:gaussBonnet,fk:scalarCurvature,cm:modularCurvature}.
However, starting from a spectral triple, with the metric implicitly
given via the Dirac operator, it is far from obvious if there exists a
module together with a bilinear form, representing the metric
corresponding to the Dirac operator, not to mention the existence of a
Levi-Civita connection. In order to better understand what kind of
results one can expect, it is interesting to take a more naive
approach, where one starts with a module together with a metric, and
tries to understand under what conditions one may discuss metric
compatibility, as well as torsion and uniqueness, of a general connection.

In \cite{aw:cgb.sphere,aw:curvature.three.sphere,w:phdthesis}, pseudo-Riemannian
calculi were introduced as a framework to discuss the existence of
a metric and torsion free connection as well as properties of its
curvature. In fact, the theory is somewhat similar to that of
Lie-Rinehart algebras, where a real calculus (as introduced in
\cite{aw:curvature.three.sphere}) might be considered as a
``noncommutative Lie-Rinehart algebra''. Lie-Rinehart algebras have
been discussed from many points of view (see
e.g. \cite{r:differential.forms,h:poisson.cohomology} and
\cite{aa:kpalgebras} for an overview of metric aspects).  Although the
existence of a Levi-Civita connection is not always guaranteed in
the context of pseudo-Riemannian calculi, it was shown that the
connection is unique if it exists. The theory has concrete
similarities with classical differential geometry, and several ideas,
such as Koszul's formula, have direct analogues in the noncommutative
setting. Apart from the noncommutative torus, noncommutative spheres
were considered, and a Chern-Gauss-Bonnet type theorem was proven for
the noncommutative 4-sphere \cite{aw:cgb.sphere}. Note that there are several approaches to
metric aspects of noncommutative geometry, and Levi-Civita
connections, which are different but similar in spirit (see
e.g. \cite{lm:lie.extensions,fgr:supersymmetric.noncommutative,ac:ncgravitysolutions,bm:starCompatibleConnections,r:leviCivita,mw:quantum.koszul}).

In this paper, we introduce morphisms of real (metric) calculi and
define noncommutative (isometric) embeddings. We show that several
basic concepts of submanifold theory extends to noncommutative
submanifolds and we prove an analogue of Gauss' equations for the
curvature of a submanifold. Moreover, the mean curvature of an
embedding is defined, immediately giving a natural definition of a
(noncommutative) minimal embedding. As an illustration of the above
concepts, the noncommutative torus is considered as a minimal
submanifold of the noncommutative 3-sphere.

\section{Pseudo-Riemannian calculi}

\noindent
Let us briefly recall the basic definitions leading to the concept
of a pseudo-Riemannian calculus and the uniqueness of the Levi-Civita
connection. For more details, we refer to \cite{aw:curvature.three.sphere}.

\begin{definition}[Real calculus]
  Let $\Alg$ be a unital $\ast$-algebra, let
  $\Der\subseteq \text{Der}(\Alg)$ be a finite-dimensional (real) Lie
  algebra and let $\Mod$ be a (right)
  $\Alg$-module. Moreover, let $\varphi: \Der\rightarrow\Mod$ be a
  $\Real$-linear map whose image generates $\Mod$ as an $\Alg$-module. Then
  $C_{\Alg}=\RCa$ is called a \textit{real calculus over} $\Alg$.
\end{definition}

\noindent
The motivation for the above definition comes from the analogous
structures in differential geometry, as seen in the following example.

\begin{example}\label{ex:Manifold.Representation}
  Let $\Mani$ be a smooth manifold. Then $\Mani$ can be represented by
  the real calculus $C_{\Alg}=(\A,\g,M,\varphi)$ with
  $\Alg=\ManiSmooth$, $\Der=\DerSet(\ManiSmooth)$, $M=\Vect(M)$ (the
  module of vector fields on $\Sigma$) and choosing $\varphi$ to be
  the natural isomorphism between the set of derivations of
  $C^\infty(\Sigma)$ and smooth vector fields on $\Mani$.
\end{example}

\noindent
Next, since we are interested in Riemannian geometry, one introduces a
metric structure on the module $M$.

\begin{definition}
  Suppose that $\Alg$ is a $\ast$-algebra and let $\Mod$ be a right
  $\Alg$-module. A \textit{hermitian form on} $\Mod$ is a map
  $h:\Mod\times\Mod\rightarrow\Alg$ with the following properties:
 \begin{enumerate}[label=$h$\arabic*.]
 \item $h(m_1,m_2+m_3)=h(m_1,m_2)+h(m_1,m_3)$ 
 \item $h(m_1,m_2a)=h(m_1,m_2)a$ 
 \item $h(m_1,m_2)=h(m_2,m_1)^*$ 
 \end{enumerate}
 for all $m_1,m_2,m_3\in M$ and $a\in \A$.  Moreover, if $h(m_1,m_2)=0$
 for all $m_2\in\Mod$ implies that $m_1=0$ then $h$ is said to be
 \textit{non-degenerate}, and in this case we say that $h$ is a
 \textit{metric on} $\Mod$. The pair $(\Mod,h)$ is called a
 \textit{(right) hermitian} $\Alg$-\textit{module}, and if $h$ is a
 metric on $\Mod$ we say that $(\Mod,h)$ is a \textit{(right) metric}
 $\Alg$-\textit{module}.
\end{definition}

\begin{definition}[Real metric calculus]
  Suppose that $C_{\Alg}=(\A,\g,M,\varphi)$ is a real calculus over $\Alg$ and that
  $(\Mod,h)$ is a (right) metric $\Alg$-module. If
  \begin{align*}
    h(\varphi(\der_1),\varphi(\der_2))^*=h(\varphi(\der_1),\varphi(\der_2)) 
  \end{align*}
  for all $\der_1,\der_2\in\Der$ then the pair $(C_{\Alg},h)$ is
  called a \textit{real metric calculus}.
\end{definition}

\begin{example}\label{ex:Pseudo.Riemannian}
  Let $(\Mani,g)$ be a Riemannian manifold and let $C_{\Alg}$ be the
  real calculus from Example~\ref{ex:Manifold.Representation}
  representing $\Mani$. Then $(C_{\Alg},g)$ is a real metric calculus.
\end{example}

\noindent
In what follows, we shall sometimes require the metric to satisfy a
stronger condition than non-degeneracy.

\begin{definition}
  Let $h$ be a metric on $\Mod$ and let
  $\hat{h}:\Mod\rightarrow \Mod^*$ (the dual of $M$) be the mapping
  given by $\hat{h}(m)(n)=h(m,n)$. The metric $h$ is said to be
  \emph{invertible} if $\hat{h}$ is invertible.
\end{definition}

\noindent
Now, given a real metric calculus $\CA=(\A,\g,M,\varphi)$, we will
discuss connections on $M$ and their compatibility with the
metric. Let us start by recalling the definition of an affine
connection for a derivation based calculus.

\begin{definition}
  Let $C_{\Alg}=(\A,\g,M,\varphi)$ be a real calculus over $\Alg$. An \textit{affine
    connection} on $(\Mod,\Der)$ is a map
  $\nabla:\Der\times\Mod\rightarrow\Mod$ satisfying
  \begin{enumerate}[label=($\arabic*$)]
  \item $\nabla_{\der}(m+n)=\nabla_{\der}m+\nabla_{\der}n$,
  \item $\nabla_{\lambda\der+\der'}m=\lambda\nabla_{\der}m+\nabla_{\der'}m$,
  \item $\nabla_{\der}(ma)=(\nabla_\der m)a+m\der(a)$
  \end{enumerate}
  for $m,n\in\Mod$, $\der,\der'\in\Der$, $a\in\Alg$ and $\lambda\in\mathbb{R}$.
\end{definition}

\noindent
The fact that we shall require the connection to be ``real'' is
reflected in the following definition.

\begin{definition}
  Let $(C_{\Alg},h)$ be a real metric calculus and let $\nabla$ denote
  an affine connection on $(\Mod,\g)$. Then $(C_{\Alg},h,\nabla)$ is
  called a \emph{real connection calculus} if
  \begin{equation*}
    h\paraa{\nabla_{\der} \varphi(\d_1), \varphi(\d_2)}=h(\nabla_{\der} \varphi(\d_1), \varphi(\d_2))^*
  \end{equation*}
  for all $\der,\d_1,\d_2\in\g$.
\end{definition}

\begin{definition}
  Let $(C_{\Alg},h,\nabla)$ be a real connection calculus. We say that
  $(C_{\Alg},h,\nabla)$ is \emph{metric} if
  \begin{equation*}
    \der(h(m,n))=h(\nabla_{\der} m, n)+ h(m,\nabla_{\der} n)	
  \end{equation*}
  for all $\der\in \g$ and $m,n\in \Mod$, and \emph{torsion-free} if
  \begin{equation*}
    \nabla_{\der_1}\varphi(\der_2)-\nabla_{\der_2}\varphi(\der_1)-\varphi([\der_1,\der_2])=0
  \end{equation*}
  for all $\der_1,\der_2\in\g$. A metric and torsion-free real connection calculus is called a \emph{pseudo-Riemannian calculus}.
\end{definition}

\noindent
A connection fulfilling the requirements of a pseudo-Riemannian
calculus is called a \emph{Levi-Civita connection}. In the quite
general setup of real metric calculi, where there are few assumptions on
the structure of the algebra $\A$ and the module $M$, the existence of
a Levi-Civita connection can not be guaranteed. However, if it exists,
it is unique.

\begin{theorem}{\emph{(\cite{aw:curvature.three.sphere})}}\label{thm:Levi.Civita.connection}
  Let $(C_{\Alg},h)$ be a real metric calculus. Then there exists at
  most one affine connection $\nabla$ such that $(C_{\Alg},h,\nabla)$
  is a pseudo-Riemannian calculus.
\end{theorem}

\noindent
The next result provides us a noncommutative analogue of Koszul's
formula, which is a useful tool for constructing the Levi-Civita
connection in several examples.

\begin{proposition}{\emph{(\cite{aw:curvature.three.sphere})}}
	Let $(C_{\Alg},h,\nabla)$ be a pseudo-Riemannian calculus and assume that $\der_1,\der_2,\der_3\in\Der$. Then
	\begin{multline}\label{Koszul}
	2h(\nabla_{1}E_2,E_3)=\der_1h(E_2,E_3)+\der_2h(E_1,E_3)-\der_3h(E_1,E_2)\\
	-\FofG{h}{E_1,\varphi([\der_2,\der_3])}+\FofG{h}{E_2,\varphi([\der_3,\der_1])}+\FofG{h}{E_3,\varphi([\der_1,\der_2])},
	\end{multline}
	where $\nabla_i=\nabla_{\der_i}$ and $E_i=\varphi(\der_i)$ for $i=1,2,3$.
\end{proposition}

\noindent
As in Riemannian geometry, a connection satisfying Koszul's formula
is torsion-free and compatible with the metric.

\begin{proposition}{\emph{(\cite{aw:curvature.three.sphere})}}\label{prop:satisfies.koszul}
  Let $(C_{\Alg},h)$ be a real metric calculus, and suppose that
  $\nabla$ is an affine connection on $(\Mod,\Der)$ such that Koszul's
  formula \emph{(\ref{Koszul})} holds. Then $(C_{\Alg},h,\nabla)$ is a
  pseudo-Riemannian calculus.
\end{proposition}

\noindent
A particularly simple case, which is also relevant to our
applications, is when $M$ is a free module. The following result then
gives a way of constructing the Levi-Civita connection from Koszul's formula.

\begin{corollary}{\emph{(\cite{aw:curvature.three.sphere})}}\label{cor:calculating.Levi.Civita}
  Let $(C_{\Alg},h)$ be a real metric calculus and let
  $\{\der_1,...,\der_n\}$ be a basis of $\Der$ such that
  $\{E_a=\varphi(\der_a)\}_{a=1}^n$ is a basis for $\Mod$. If there
  exist $m_{ab}\in\Mod$ such that
  \begin{multline}\label
    2h(m_{ab},E_c)=\der_a h(E_b,E_c)+\der_b h(E_a,E_c)-\der_c h(E_a,E_b)\\
    -\FofG{h}{E_a,\varphi([\der_b,\der_c])}+\FofG{h}{E_b,\varphi([\der_c,\der_a])}+\FofG{h}{E_c,\varphi([\der_a,\der_b])},
  \end{multline}
  for $a,b,c=1,...,n$, then there exists an affine connection
  $\nabla$, given by $\nabla_{\der_a}E_b=m_{ab}$, such that
  $(C_{\Alg},h,\nabla)$ is a pseudo-Riemannian calculus.
\end{corollary}

\section{Real calculus homomorphisms}

\noindent
In order to understand the algebraic structure of real calculi, a
first step is to consider morphisms. Via a concept of morphism of real
calculi, one can understand when two calculi are considered to be
equal (isomorphic) and, from a geometric point of view, what one means
by a noncommutative embedding. In this section we introduce
homomorphisms of real (metric) calculi and prove several results
which, in different ways, shed light on the new concept.

\begin{definition}\label{DefHom}
  Let $C_{\Alg}=\RCa$ and $C_{\Algprim}=\RCb$ be real calculi and
  assume that $\phi:\Alg \rightarrow \Algprim$ is a $\ast$-algebra
  homomorphism. If there is a map $\psi:\Derprim \rightarrow \Der$
  such that
  \begin{enumerate}[label=($\psi$\arabic*)]
  \item  $\psi$ is a Lie algebra homomorphism
  \item \label{Compatibility} $\delta(\phi(a))=\phi(\psi(\delta)(a))$ for all $\delta\in\Derprim, a\in\Alg$,
  \end{enumerate}
  then $\psi$ is said to be \textit{compatible} with $\phi$. If
  $\psi$ is compatible with $\phi$ we define $\Psi$ as
  $\Psi=\varphi\circ\psi$, and $\Mod_{\Psi}$ is defined to be the
  submodule of $\Mod$ generated by $\Psi(\Derprim)$.
  
  Furthermore, if there is a map
  $\ModHom:\Mod_{\Psi}\rightarrow\Modprim$ such that
  \begin{enumerate}[label=($\ModHom$\arabic*)]
  \item $\ModHom(m_1+m_2)=\ModHom(m_1)+\ModHom(m_2)$ for all
    $m_1,m_2\in\Mod$\label{It1}
  \item $\ModHom(ma)=\ModHom(m)\phi(a)$ for all $m\in\Mod$ and $a\in\Alg$\label{It2}
  \item $\ModHom(\Psi(\delta))=\varphi'(\delta)$ for all $\delta\in\Derprim$\label{It3},
  \end{enumerate}
  then $\ModHom$ is said to be \textit{compatible} with $\phi$ and
  $\psi$, and $\Hom$ is called a \textit{real calculus homomorphism}
  from $C_{\Alg}$ to $C_{\Algprim}$ (see Figure~\ref{RC-Hom2} for an
  illustration of a real calculus homomorphism). If $\phi$ is a
  $\ast$-algebra isomorphism, $\psi$ a Lie algebra isomorphism and
  $\psih$ is a bijective map then $(\phi,\psi,\psih)$ is called a real
  calculus isomorphism.
\end{definition}

\begin{figure}[h]
\begin{center}
  \begin{tikzpicture}
  \draw (0,1) node{$C_{\Alg}$};
  \draw (4,1) node{$C_{\Algprim}$};
  	\fill[fill=green!0!white, rounded corners] (-1,0.7) rectangle (1,-4.7);
    \fill[fill=green!0!white, rounded corners] (3,0.7) rectangle (5,-4.7);
  
    \filldraw[fill=black!20!white, draw=white!50!black] (0,0) ellipse [x radius=20pt, y radius=10pt];
    \draw (0,0) node{$\Alg$};
    
    \filldraw[fill=black!0!white, draw=white!0!black, rounded corners] (-0.7,-1.5) rectangle (0.7,-2.5);
    \draw (0.5,-1.3) node{$\Der$};
    \filldraw[fill=black!20!white, draw=white!50!black] (0,-2) ellipse [x radius=15pt, y radius=10pt];
    \draw (0,-2) node{$\psi(\Derprim)$};
    
    \filldraw[fill=black!20!white, draw=white!50!black] (0,-4) ellipse [x radius=20pt, y radius=10pt];
    \filldraw[fill=black!0!white, draw=white!0!black, rounded corners] (-0.7,-3.5) rectangle (0.7,-4.5);
    \draw (0.5,-3.3) node{$\Mod$};
    \filldraw[fill=black!20!white, draw=white!50!black] (0,-4) ellipse [x radius=15pt, y radius=10pt];
    \draw (0,-4) node{$\Mod_{\Psi}$};

	\filldraw[fill=black!20!white, draw=white!50!black] (4,0) ellipse [x radius=20pt, y radius=10pt];
    \draw (4,0) node{$\Algprim$};
    
    \filldraw[fill=black!20!white, draw=white!50!black] (4,-2) ellipse [x radius=20pt, y radius=10pt];
    \draw (4,-2) node{$\Derprim$};
    
    \filldraw[fill=black!20!white, draw=white!50!black] (4,-4) ellipse [x radius=20pt, y radius=10pt];
    \draw (4,-4) node{$\Modprim$};
    
    \draw[->] (0.75,0) .. controls (1.5,0.2) and (2.5,0.2) ..
    node[above=2pt] {$\phi$} (3.25,0);
    
    \draw[<-] (0.75,-2) .. controls (1.5,-1.8) and (2.5,-1.8) ..
    node[above=2pt] {$\psi$} (3.25,-2);
    
    \draw[->] (0.55,-4) .. controls (1.5,-3.8) and (2.5,-3.8) ..
    node[above=2pt] {$\ModHom$} (3.25,-4);
    
    \draw[->] (0,-2.5) .. controls (-0.1,-2.7) and (-0.1,-3.3) .. 
    node[left=1pt] {$\varphi$} (0,-3.5);
    
    \draw[->] (4,-2.4) .. controls (4.1,-2.7) and (4.1,-3.3) .. 
    node[right=1pt] {$\varphi '$} (4,-3.6);
    
    \draw[<-] (0,-0.4) .. controls (-0.1,-0.7) and (-0.1,-1.3) .. 
    (0,-1.5);

    \draw[<-] (4,-0.4) .. controls (4.1,-0.7) and (4.1,-1.3) .. 
    (4,-1.6);
  
  \end{tikzpicture}
\end{center}
\caption{A real calculus homomorphism $(\phi,\psi,\psih):C_{\A}\to C_{\A'}$.}
\label{RC-Hom2}
\end{figure}

Let us try to understand Definition~\ref{DefHom} in the context of
embeddings, where the analogy with classical geometry is rather
clear. Thus, let $\phi_0:\Sigma'\to\Sigma$ be an embedding of
$\Sigma'$ into $\Sigma$ and let
$\phi:C^\infty(\Sigma)\to\C^{\infty}(\Sigma')$ be the corresponding
homomorphism of the algebras of smooth functions. In the notation of
Definition~\ref{DefHom} we have
\begin{alignat*}{2}
  &\A = C^\infty(\Sigma) & \quad\overset{\phi}{\longrightarrow}\quad &\A'=C^\infty(\Sigma')\\
  &\g  = \DerSet(\A) & \quad\overset{\psi}{\longleftarrow}\quad &\g' = \DerSet(\A')\\
  M = &\Vect(\Sigma)\supseteq\MPsi & \quad\overset{\psih}{\longrightarrow}\quad &M'=\Vect(\Sigma').
\end{alignat*}
First of all, there is no natural map from $\Vect(\Sigma)$ to
$\Vect(\Sigma')$ since a vector field $X\in \Vect(\Sigma)$ at a
point $p\in\phi_0(\Sigma')$ might not lie in $T_p\Sigma'$ (regarded
as a subspace of $T_p\Sigma$). However, vector fields which are
tangent to $\Sigma'$ in this sense may be restricted to
$\Sigma'$. On the other hand, any vector field $X'\in\Vect(\Sigma')$
(assuming $\Sigma'$ to be closed) can be extended to a smooth vector
field $X\in\Vect(\Sigma)$ such that $X|_{\Sigma'}=X'$. In light of
the isomorphism between vector fields and derivations, it is
therefore more natural to have a map
$\psi:\DerSet(\A')\to\DerSet(\A)$, corresponding to a choice of
extension of vector fields on $\Sigma'$. The map $\psih$ then
corresponds to the restriction of vector fields on $\Sigma$ which
are tangent to $\Sigma'$. Consequently, we consider vector fields in
$\MPsi$ as extensions of vector fields on the embedded manifold.

In noncommutative geometry (in contrast to the classical case) $\g$
is no longer an $\A$-module, a difference which is captured by the
concept of a real calculus. The definition of homomorphism reflects
this fact by assuming that every derivation of $\A'$ can be
``extended'' to a derivation of $\A$ and, furthermore, that every
vector field on $\Sigma$ which is tangent to $\Sigma'$ (that is, in
the image of $\varphi\circ\psi$) can be ``restricted'' to $\Sigma'$.

Next, one can easily check that the composition of two homomorphisms
is again a homomorphism.

\begin{proposition}\label{prop:composition}
  Let $C_{\Alg}$, $C_{\Algprim}$ and $C_{\Alg''}$ be real calculi and
  assume that
  \begin{align*}
    (\phi,\psi,\psih):C_{\A}\to C_{\A'}\mathand
    (\phi',\psi',\psih'):C_{\A'}\to C_{\A''}
  \end{align*}
  are real calculus homomorphisms. Then
  $(\phi'\circ\phi,\psi\circ\psi',\psih'\circ\psih):C_\A\to\C_{\A''}$
  is a real calculus homomorphism.
\end{proposition}

\begin{proof}
  For convenience, we introduce $\Phi:=\phi'\circ\phi$,
  $\tilde{\psi}:=\psi\circ\psi'$ and $\hat{\Psi}:=\psih'\circ\psih$.
  First of all, it is clear that $\Phi$ is a $\ast$-algebra homomorphism and $\tilde{\psi}$ is a
  Lie algebra homomorphism. For $a\in \Alg$ and
  $\delta\in \g''$ we get that
  \begin{equation*}
    \delta(\Phi(a))=\phi'(\psi'(\delta)(\phi(a)))=\phi'(\phi(\tilde{\psi}(\delta)(a)))=\Phi(\tilde{\psi}(\delta)(a)),
  \end{equation*}
  showing that $\Phi$ and $\tilde{\psi}$ are compatible, with
  $M_{\tilde{\Psi}}$ being the submodule of $\Mod$ generated by
  $\tilde{\psi}(\g'')$.
  Checking that $\hat{\Psi}(m+n)=\hat{\Psi}(m)+\hat{\Psi}(n)$ and
  $\hat{\Psi}(ma)=\hat{\Psi}(m)\Phi(a)$ for all
  $m,n\in M_{\tilde{\Psi}}$ and $a\in \Alg$ is trivial, and for
  $\delta\in \g''$ we get
  \begin{equation*}
    \varphi''(\delta)=\psih'(\Psi'(\delta))=\psih'(\varphi'(\psi'(\delta)))=\psih'(\psih(\Psi(\psi'(\delta))))=\hat{\Psi}(\varphi\circ\tilde{\psi}(\delta)).
  \end{equation*}
  Thus $\hat{\Psi}$ is compatible with $\Phi$ and $\tilde{\psi}$, and
  it follows that $(\Phi,\tilde{\psi},\hat{\Psi})$ is a real calculus
  homomorphism from $C_{\Alg}$ to $C_{\Alg''}$.
\end{proof}

\noindent
A homomorphism of real calculi $(\phi,\psi,\psih)$ consists of three
maps, and a natural question is what kind of freedom one has in
choosing these maps? Let us start by showing that, given $\phi$ and
$\psi$, there is at most one $\psih$ such that $(\phi,\psi,\psih)$ is a real calculus homomorphism.

\begin{proposition}\label{ModHomUnique}
  If $(\phi,\psi,\psih)$ and $(\phi,\psi,\psit)$ are real calculus
  homomorphisms from $\CA$ to $\CAp$ then $\psih=\psit$.
\end{proposition}

\begin{proof}
  Let $m=\Psi(\delta_i)a^i$ for $\delta_i\in\g'$ and $a^i\in\A$ be an arbitrary element of $\MPsi$.
  It follows from \ref{It1}-\ref{It3} that
  \begin{align*}
    \tilde{\psi}(m)&=\tilde{\psi}(\Psi(\delta_i)a^i)=\tilde{\psi}(\Psi(\delta_i))\phi(a^i)=\varphi'(\delta_i)\phi(a^i)=\ModHom(\Psi(\delta_i))\phi(a^i)\\
                   &=\ModHom(\Psi(\delta_i)a^i)=\ModHom(m).\qedhere
  \end{align*} 
\end{proof}

\noindent
Furthermore, if $\phi$ is an isomorphism, then the next result shows
that $\psi$ is determined uniquely by $\phi$. Thus, combined with the
previous result we conclude that if $(\phi,\psi,\psih)$ is an
isomorphism of real calculi, then $\psi$ and $\psih$ are uniquely
determined by $\phi$.

\begin{proposition}\label{prop:RCIsoLemma}
  If $(\phi,\psi,\psih):\CA\to\CAp$ is a real calculus homomorphism
  such that $\phi$ is an isomorphism, then $\psi$ is a Lie algebra isomorphism
  with
  \begin{align*}
    \psi(\delta)=\phi^{-1}\circ\delta\circ\phi
  \end{align*}
  for $\delta\in\g'$. 
\end{proposition}

\begin{proof}
  The formula for $\psi$ follows directly from the fact that
  $\delta(\phi(a))=\phi(\psi(\delta)(a))$ together with $\phi$ being
  an isomorphism. To prove that $\psi$ is an isomorphism, let
  $\tilde{\psi}:\Der\rightarrow\Derprim$ be given by
  $\tilde{\psi}(\der)=\phi\circ\der\circ\phi^{-1}$.  Then for any
  $\der\in\Der$ and $\delta\in\Derprim$ it follows that
  \begin{align*}
    \psi\circ\tilde{\psi}(\der)&=\phi^{-1}\circ\tilde{\psi}(\der)\circ\phi=\phi^{-1}\circ\phi\circ\der\circ\phi^{-1}\circ\phi=\der\\
    \tilde{\psi}\circ\psi(\delta)&=\phi\circ\psi(\delta)\circ\phi^{-1}=
                                   \phi\circ\phi^{-1}\circ\delta\circ\phi\circ\phi^{-1}=\delta.
  \end{align*}
  Thus $\psi$ is a bijection with inverse $\psi^{-1}=\tilde{\psi}$.
  Furthermore, $\psi^{-1}$ preserves the Lie bracket:
  \begin{align*}
    \psi^{-1}([\der_1,\der_2])&=\psi^{-1}([\psi\circ\psi^{-1}(\der_1),\psi\circ\psi^{-1}(\der_2)])=\psi^{-1}\circ\psi([\psi^{-1}(\der_1),\psi^{-1}(\der_2)])\\
                              &=[\psi^{-1}(\der_1),\psi^{-1}(\der_2)],
  \end{align*}
  proving that $\psi$ is indeed a Lie algebra isomorphism.
\end{proof} 

\noindent
Given a homomorphism $(\phi,\psi,\psih):\CA\to\CAp$, there is a
natural $\A$-module structure on $M'$ given by $m'\cdot a=m'\phi(a)$
for $m'\in M'$ and $a\in\A$. As expected, the right $\A$-modules $M$
and $M'$ are isomorphic when $(\phi,\psi,\psih)$ is an isomorphism.

\begin{proposition}\label{prop:RC.Mod.Isomorphic}
  If $(\phi,\psi,\psih):\CA\to\CAp$ is a real calculus isomorphism
  then 
  \begin{align*}
    \Mod=\Mod_{\Psi}\simeq \Modprim.
  \end{align*}
\end{proposition}

\begin{proof}
  Since $\psi$ is an isomorphism it follows that $\Der=\psi(\Derprim)$. From
  this it immediately follows that $\Mod=\Mod_{\Psi}$, since $\Mod_{\Psi}$ is
  defined to be the submodule of $\Mod$ generated by $\Der=\psi(\Derprim)$.
  Considering $M'$ as a right $\A$-module, $\psih$ is an $\A$-module
  homomorphism, and since $\psih$ is assumed to be bijective, we
  conclude that $\MPsi\simeq M'$.
\end{proof}

\noindent
Recalling our previous discussions of real calculus homomorphisms in
relation to embeddings, one may consider vector fields in $\MPsi$ as
extensions of vector fields in $M$. Let us therefore make the
following definition.

\begin{definition}
  If $m\in\MPsi$ such that $\psih(m)=m'$ then $m$ is called an
  \emph{extension of $m'$}. The set of extensions of $m'$ will be
  denoted by $\ExtPsi(m')$.
\end{definition}

\noindent
\subsection{Homomorphisms of real metric calculi}

\noindent Having introduced the concept of homomorphisms for real
calculi, it is natural to proceed to real metric calculi. From the
geometric point of view, in the case of embeddings, one would like a
homomorphism of real metric calculi to correspond to an isometric
embedding. The following definition is straightforward.

\begin{definition}
  Let $(\CA,h)$ and $(\CAp,h')$ be real metric calculi and assume that
  $(\phi,\psi,\psih):\CA\to\CAp$ is a real calculus homomorphism. If
  \begin{equation*}
    h'\paraa{\varphi'(\delta_1),\varphi'(\delta_2)}=\phi\paraa{
      h (\Psi (\delta_1), \Psi(\delta_2))}
  \end{equation*}
  for all $\delta_1,\delta_2\in\Derprim$ then 
  $(\phi,\psi,\psih)$ is called a \emph{real metric
    calculus homomorphism}.
\end{definition}

\noindent
Assume that $(\phi,\psi,\psih):(\CA,h)\to\CA$ is a homomorphism of
real calculi. It is natural to ask if there exists a metric $h'$
such that $(\phi,\psi,\psih):(\CA,h)\to(\CA,h')$ is a homomorphism of
real metric calculi, in which case we would call $h'$ the
\emph{induced} metric. As it turns out, one cannot guarantee the
existence of $h'$, but whenever it exists, it is unique; we state this
as follows.

\begin{proposition}\label{IndMetric}
  Let $C_{\A}$ be a real calculus, $(C_{\Alg},h)$ a real metric
  calculus, and let $(\phi,\psi,\ModHom):(C_{\A},h)\to C_{\A'}$ be a
  real calculus homomorphism. Then there exists at most one hermitian form
  $h'$ on $\Modprim$ satisfying
  \begin{equation*}
    h'(\varphi'(\delta_1),\varphi'(\delta_2))=\phi\paraa{
    h (\Psi (\delta_1), \Psi (\delta_2))},
    \quad\delta_1,\delta_2\in\Derprim.
  \end{equation*}
\end{proposition}

\begin{proof}
  Suppose that $h_1'$ and $h_2'$ both fulfill the given conditions for
  $h'$.  By definition of real calculus homomorphism it is immediately
  obvious that $h_1'$ and $h_2'$ agree on $\varphi'(\Derprim)$. If we
  take two arbitrary elements $m,n\in\Modprim$ it follows from the
  fact that $C_{\Algprim}$ is a real calculus over $\Algprim$ that $m$
  and $n$ can be written as
  \begin{align*} m'&= \varphi'(\delta_i)a^i,\quad \delta_i\in\Derprim,
                     a^i\in\Algprim,\\ n'&= \varphi'(\delta_j)b^j,\quad
                                           \delta_j\in\Derprim, b^j\in\Algprim.
  \end{align*}
  Furthermore, one obtains
  \begin{align*}
    h_1'(m',n')&=\FofG{h_1'}{\varphi'(\delta_i)a^i,\varphi'(\delta_j)b^j}=\FofG{h_1'}{\varphi'(\delta_i)a^i,\varphi'(\delta_j)}b^j\\
               &=(a^i)^*\FofG{h_1'}{\varphi'(\delta_i),\varphi'(\delta_j)}b^j=(a^i)^*\FofG{h_2'}{\varphi'(\delta_i),\varphi'(\delta_j)}b^j\\
               &=\FofG{h_2'}{\varphi'(\delta_i)a^i,\varphi'(\delta_j)}b^j=\FofG{h_2'}{\varphi'(\delta_i)a^i,\varphi'(\delta_j)b^j}=h_2'(m',n'),
  \end{align*} since $h_1'$ and $h_2'$ are hermitian forms on $M'$ and
  $h_1'(\varphi'(\delta_i),\varphi'(\delta_j))=h_2'(\varphi'(\delta_i),\varphi'(\delta_j))$
  for $\delta_1,\delta_2\in\Derprim$. Since $m'$ and $n'$ are
  arbitrary, it follows that $h_1'=h_2'$.
\end{proof}

\noindent Note that if $(\phi,\psi,\psih):(\CA,h)\to(\CAp,h')$ is a
homomorphism of real metric calculi, then
$\phi\paraa{h(m,n)}=h'(\psih(m),\psih(n))$ for all $m,n\in\MPsi$. In other
words
\begin{align*}
  \phi\paraa{h(m,n)} = h'(m',n')
\end{align*}
if $m\in\ExtPsi(m')$ and $n\in\ExtPsi(n')$. This is to be compared
with the geometrical situation where the inner product of vector
fields restricted to the isometrically embedded manifolds equals the
inner product of the restricted vector fields.

\section{Embeddings of real calculi}

\noindent In the previous section, we highlighted the analogy with
embedded manifolds in order to motivate and understand the different
concepts introduced for noncommutative algebras. However, we did not
make the distinction between general homomorphisms and embeddings
precise. In this section we shall define noncommutative embeddings and
introduce a theory of submanifolds, much in analogy with the classical
situation. It turns out that one can readily introduce
the second fundamental form, and find a noncommutative analogue of
Gauss' equation, giving the curvature of the submanifold.

A necessary condition for a map $\phi_0:\Sigma'\to\Sigma$ to be an
embedding, is that $\phi_0$ is injective; dually, this corresponds to
$\phi:C^{\infty}(\Sigma)\to C^\infty(\Sigma')$ being surjective. To
formulate the next definition, we recall the orthogonal complement of a
module. Namely, let $(\CA,h)$ be a real metric calculus. Given any
subset $N\subseteq M$, we define $N^\perp = \{m\in M: h(m,n)=0\}$ and
note that $N^\perp$ is a $\A$-module.

\begin{definition}\label{def:embedding}
  A homomorphism of real calculi $(\phi,\psi,\psih):\CA\to\CAp$ is
  called an \emph{embedding} if $\phi$ is surjective and there exists
  a submodule $\tilde{M}\subseteq M$ such that
  $M=\MPsi\oplus\tilde{M}$. A homomorphism of real metric calculi
  $(\phi,\psi,\psih):(\CA,h)\to(\CAp,h')$ is called an \emph{isometric
    embedding} if $(\phi,\psi,\psih)$ is an embedding and
  $M=\MPsi\oplus\MPsi^\perp$.
\end{definition}

\noindent
The surjectivity of $\phi$ has immediate implications for the maps
$\psi$ and $\psih$.

\begin{proposition}\label{RCHomSurj}
  Assume that $\Hom:C_{\A}\to C_{\A'}$ is a real calculus homomorphism
  such that that $\phi$ is surjective. Then $\psi$ is injective and
  $\ModHom$ is surjective.
\end{proposition}

\begin{proof}
  For the first statement, suppose
  $\delta\in\mbox{ker}(\psi)$. Then for any $a\in\Alg$ it follows that
  $\psi(\delta)(a)=0$. Thus, by \ref{Compatibility} it follows that
  \begin{equation*}
    \delta(\phi(a))=\phi(\psi(\delta)(a))=\phi(0)=0
  \end{equation*}
  for any $a\in\Alg$, and since $\phi$ is surjective it follows that
  $\delta(a')=0$ for every $a'\in\Algprim$.
  
  For the second statement, let $m'\in\Modprim$. Then $m'$ can be
  written on the form $m'=\varphi'(\delta_i)b^i$ for some
  $\delta_i\in\Derprim$ and $b^i\in\Algprim$, and since $\phi$ is
  surjective there are $a^i\in\Alg$ such that $\phi(a^i)=b^i$. It
  follows that
  \begin{equation*}
    m'=\varphi'(\delta_i)b^i=\ModHom(\Psi(\delta_i))\phi(a^i)=\FofG{\ModHom}{\Psi(\delta_i)a^i},
  \end{equation*}
  completing the proof.
\end{proof}

\noindent
Note that Proposition~\ref{RCHomSurj} gives further motivation for
Definition~\ref{def:embedding} since it shows that $\psi$ is
injective, in analogy with the injectivity of the tangent map of an
embedding. Moreover, it follows from Proposition~\ref{RCHomSurj} that if
$(\phi,\psi,\psih):\CA\to\CAp$ is an embedding, then every element
$m'\in M'$ has at least one extension corresponding to the geometric
situation where a vector field on the embedded manifold can be
extended to a vector field in the ambient space.

Furthermore, given an embedding $(\phi,\psi,\psih):\CA\to\CAp$, we
define the $\A$-linear projection $P:M\to\MPsi$ as
\begin{align*}
  P(m_\Psi\oplus \tilde{m}) = m_{\Psi}
\end{align*}
with respect to the decomposition $M=\MPsi\oplus\tilde{M}$. The
complementary projection will be denoted by $\Pi=\mid-P$. (Note that
for an embedding of real metric calculi, the projections $P$ and $\Pi$
are orthogonal with respect to the metric on $M$.)

In analogy with classical Riemannian submanifold theory (see
e.g. \cite{kn:foundationsDiffGeometryII}), one decomposes the
Levi-Civita connection in its tangential and normal parts. Let
$(\CA,h,\nabla)$ and $(\CAp,h',\nabla')$ be pseudo-Riemannian calculi
and assume that $(\phi,\psi,\psih):(\CA,h)\to(\CAp,h')$ is an
isometric embedding and write
\begin{align}
  &\nabla_{\psi(\delta)}m = L(\delta,m) + \alpha(\delta,m)\label{eq:gauss.eq}\\
  &\nabla_{\psi(\delta)}\xi = -A_\xi(\delta) + D_{\delta}\xi\label{eq:weingarten.eq}
\end{align}
for $\delta\in\g'$, $m\in\MPsi$ and $\xi\in\MPsi^\perp$, with
\begin{alignat*}{2}
  &L(\delta,m) = P(\nabla_{\psi(\delta)}m) &\qquad &\alpha(\delta,m) = \Pi(\nabla_{\psi(\delta)}m)\\
  &A_\xi(\delta) = -P(\nabla_{\psi(\delta)}\xi) & &D_{\delta}\xi = \Pi(\nabla_{\psi(\delta)}\xi).
\end{alignat*}

\noindent
In differential geometry, \eqref{eq:gauss.eq} is called \emph{Gauss'
  formula} and \eqref{eq:weingarten.eq} is called \emph{Weingarten's
  formula}. Furthermore, $\alpha:\g'\times \MPsi\to\MPsi^\perp$ is
called the \emph{second fundamental form} and
$A:\g'\times\MPsi^\perp\to\MPsi$ is called the \emph{Weingarten map}.
Let us start by showing that the tangential part $L(\delta,m)$ is an
extension of the Levi-Civita connection on $(\CAp,h',\nabla')$.

\begin{proposition}\label{prop:ConnectionExtension}
  If $\delta\in\g'$ and $m\in\ExtPsi(m')$ then $L(\delta,m)\in\ExtPsi(\nabla'_{\delta}m')$
\end{proposition}

\begin{proof}
  For the sake of readability, let us first establish some notation.  Let
  $\delta_i\in \g'$ and let
  $\der_i=\psi(\delta_i), E_i=\Psi(\delta_i)$ and
  $E_i'=\varphi'(\delta_i)$. Moreover, let $h_{ij}=h(E_i,E_j)$ and let
  $h_{i,[j,k]}=h(E_i,\Psi([\delta_j,\delta_k]))$; likewise, let
  $h'_{ij}=h(E_i',E_j')$ and
  $h'_{i,[j,k]}=h(E_i',\varphi'([\delta_j,\delta_k]))$.
	
  With this notation in place, Koszul's formula yields
  \begin{align*}
    2h(\nabla_iE_j,E_k)&=\der_i h_{jk}+\der_j h_{ik}-\der_k h_{ij}-h_{i,[j,k]}+h_{j,[k,i]}+h_{k,[i,j]}\\
    2h'(\nabla'_iE_j',E_k')&=\delta_i h'_{jk}+\delta_j h'_{ik}-\delta_k h'_{ij}-h'_{i,[j,k]}+h'_{j,[k,i]}+h'_{k,[i,j]}
  \end{align*}
  for all $\delta_i,\delta_j,\delta_k\in \g'$, and since $h'$ is
  induced from $h$ it follows that
  \begin{align*}
    &h'_{jk}=\phi(h_{jk})\\
    &h'_{i,[j,k]}=\phi(h_{i,[j,k]})\\
    &\delta_i h'_{jk}=\delta_i\phi(h_{jk})=\phi(\der_i(h_{jk}));
  \end{align*}
  from this it becomes clear that
  $h'(\nabla'_iE_j',E_k')=\phi(h(\nabla_iE_j,E_k))$.  Let
  $m=E_ia^i\in\MPsi$ and $n=E_kb^k\in\MPsi$ be arbitrary elements in
  $\MPsi$, where $a^i, b^k\in \Alg$. By definition of affine
  connections it follows that
  \begin{align*}
    h(\nabla_j m,n)&=h(\nabla_j(E_ia^i),E_kb^k)=h((\nabla_jE_i)a^i,E_kb^k)+h(E_i\der_j(a^i),E_kb^k)\\
                   &=(a^i)^*h(\nabla_jE_i,E_k)b^k+\der_j(a^i)^*h_{ik}b^k,
  \end{align*}
  and we get
  \begin{align*}
    \phi(h(\nabla_j m,n))&=\phi(a^i)^*h'(\nabla'_jE'_i,E'_k)\phi(b^k)+\phi(\der_j(a^i)^*)h'_{ik}\phi(b^k)\\
                         &=\phi(a^i)^*h'(\nabla'_jE'_i,E'_k)\phi(b^k)+\delta_j(\phi(a^i)^*)h'_{ik}\phi(b^k)\\
                         &=h'((\nabla'_jE'_i)\phi(a^i),E'_k\phi(b^k))+h'(E'_i\delta_j(\phi(a^i)),E'_k\phi(b^k))\\
                         &=h'(\nabla'_j(E'_i\phi(a^i)),E'_k\phi(b^k))=h'(\nabla'_j(\psih(m)),\psih(n)).
  \end{align*}
  It now follows that
  \begin{equation*}
    h'(\nabla'_j(\psih(m)),\psih(n))=\phi(h(\nabla_j m,n))=\phi(h(P(\nabla_j m),n))=\phi(h(L(\delta_j,m),n)),
  \end{equation*}
  which equals $h'(\psih(L(\delta_j,m)),\psih(n))$. Thus,
  \begin{equation*}
    h'(\nabla'_j(\psih(m)),\psih(n))=h'(\psih(L(\delta_j,m)),\psih(n)),
  \end{equation*}
  and since $h'$ is non-degenerate and $\psih$ is surjective, it
  follows that $\psih(L(\delta_j,m))=\nabla'_j \psih(m)$ which is
  equivalent to $L(\delta_j,m)\in \ExtPsi(\nabla'_j \psih(m))$, and it
  immediately follows that if $m\in \ExtPsi(m')$ then
  $L(\delta,m)\in\ExtPsi(\nabla_{\delta}m')$ for any $\delta\in \g'$
  and $m'\in\Modprim$.
\end{proof}

\noindent In view of the above result, we introduce the notation
$L(\delta,m)=\nablah'_\delta m$ and conclude that
\begin{align*}
  \nabla'_\delta m' = \psih\paraa{\nablah'_\delta m} =\psih\paraa{P(\nabla_{\psi(\delta)}m)}
\end{align*}
if $m\in\ExtPsi(m')$, giving a convenient way of retrieving the
Levi-Civita connection $\nabla'$ from $\nabla$.  Next, let us show
that the second fundamental form shares the properties of its
classical counterpart.

\begin{proposition}
  If $\delta_1,\delta_2\in\g'$, $a_1,a_2\in \Alg$ and $\lambda_1,\lambda_2\in\reals$ then
  \begin{align*}
    &\alpha\paraa{\delta_1,\Psi(\delta_2)}=\alpha\paraa{\delta_2,\Psi(\delta_1)}\\
    &\alpha\paraa{\lambda_1\delta_1+\lambda_2\delta_2,m_1}
      = \lambda_1\alpha(\delta_1,m_1)+\lambda_2\alpha(\delta_2,m_1)\\
    &\alpha(\delta_1,m_1a_1+m_2a_2)
      = \alpha(\delta_1,m_1)a_1+\alpha(\delta_1,m_2)a_2
  \end{align*}
  for $m_1,m_2\in\MPsi$.
\end{proposition}

\begin{proof}
  For the first statement, let
  $\Delta(\delta_1,\delta_2)=\alpha(\delta_1,\Psi(\delta_2))-\alpha(\delta_2,\Psi(\delta_1))$. With
  this notation in place one may use the fact that $\nabla$ is
  torsion-free to get:
  \begin{align*}
    0&=\nabla_{\psi(\delta_1)} \Psi(\delta_2)-\nabla_{\psi(\delta_2)} \Psi(\delta_1)-\varphi([\psi(\delta_1),\psi(\delta_2)])\\
     &=\nabla_{\psi(\delta_1)} \Psi(\delta_2)-\nabla_{\psi(\delta_2)} \Psi(\delta_1)-\Psi([(\delta_1),(\delta_2)])\\
     &=P(\nabla_{\psi(\delta_1)} \Psi(\delta_2))-P(\nabla_{\psi(\delta_2)} \Psi(\delta_1))-\Psi([\delta_1,\delta_2])+\Delta(\delta_1,\delta_2),
  \end{align*}
  and since the projection $P$ is linear, together with the fact that
  $P(\Psi([\delta_1,\delta_2]))=\Psi([\delta_1,\delta_2])\in \MPsi$,
  it follows that
  \begin{align*}
    0=P(\nabla_{\psi(\delta_1)} &\Psi(\delta_2)-\nabla_{\psi(\delta_2)} \Psi(\delta_1)-\Psi([\delta_1,\delta_2]))+\Delta(\delta_1,\delta_2)\\
    &=P(0)+\Delta(\delta_1,\delta_2)=0+\Delta(\delta_1,\delta_2)=\Delta(\delta_1,\delta_2).
  \end{align*}
  For the second and third statements we use the linearity of the
  connection:
  \begin{align*}
    \alpha\paraa{\lambda_1\delta_1+\lambda_2\delta_2,m_1}&=(\One-P)\paraa{\nabla_{\psi(\lambda_1\delta_1+\lambda_2\delta_2)} m_1}\\
                                                         &=(\One-P)\paraa{\lambda_1\nabla_{\psi(\delta_1)} m_1+\lambda_2\nabla_{\psi(\delta_2)} m_1}\\
                                                         &=\lambda_1\alpha\paraa{\delta_1,m_1}+\lambda_2\alpha\paraa{\delta_2,m_1}
  \end{align*}
  and
  \begin{align*}
    \alpha(\delta_1,m_1a_1+m_2a_2)&=(\One-P)\para{\nabla_{\delta_1} (m_1a_1+m_2a_2)}\\
                                  &=(\One-P)\para{\nabla_{\delta_1} (m_1a_1)+\nabla_{\delta_1}(m_2a_2)}\\
                                  &=\alpha(\delta_1,m_1a_1)+\alpha(\delta_1,m_2a_2).
  \end{align*}
  Noting that
  \begin{align*}
    \alpha(\delta_1,m_1a_1)&=\nabla_{\psi(\delta_1)} m_1a_1 - P(\nabla_{\psi(\delta_1)} m_1a_1)\\
                           &=(\nabla_{\psi(\delta_1)} m_1)a_1 +m_1\psi(\delta_1)(a_1) - P\paraa{(\nabla_{\psi(\delta_1)} m_1)a_1 +m_1\psi(\delta_1)(a_1)}\\
                           &=(\nabla_{\psi(\delta_1)} m_1)a_1 +m_1\psi(\delta_1)(a_1) - P(\nabla_{\psi(\delta_1)} m_1)a_1 -m_1\psi(\delta_1)(a_1)\\
                           &=(\nabla_{\psi(\delta_1)} m_1 - P(\nabla_{\psi(\delta)} m_1))a_1=\alpha(\delta_1,m_1)a_1
  \end{align*}
  and (similarly) that $\alpha(\delta_1,m_2a_2)=\alpha(\delta_1,m_2)a_2$ the proposition now follows.
\end{proof}

\begin{proposition}
  If $\delta\in\g'$, $m\in\MPsi$ and $\xi\in\MPsi^\perp$ then
  \begin{align*}
    h\paraa{A_\xi(\delta),m} = h\paraa{\xi,\alpha(\delta,m)}.
  \end{align*}
\end{proposition}

\begin{proof}
  Since $h(m,\xi)=0$ one can use that $(C_{\Alg},h,\nabla)$ is metric
  to see that
  $0=\psi(\delta)(h(m,\xi))=
  h(\nabla_{\psi(\delta)}\xi,m)+h(\xi,\nabla_{\psi(\delta)}m)$. Using
  that $P$ is an orthogonal projection, it follows that
  \begin{align*}
    h(A_\xi(\delta),m)&=-h(P(\nabla_{\psi(\delta)}\xi),m)\\
                      &=-h(\nabla_{\psi(\delta)}\xi,m)=h(\xi,\nabla_{\psi(\delta)}m)=h\paraa{\xi,\alpha(\delta,m)}
  \end{align*}
  as desired.
\end{proof}

\noindent
Having considered properties of $L$, $\alpha$ and $A_\xi$, let us now
show that $D_X$ has the properties of an affine connection; in
differential geometry, $D_X$ is usually identified with a connection
on the normal bundle of the submanifold.

\begin{proposition}
  If $\delta_1,\delta_2\in \g'$, $\xi_1,\xi_2\in\MPsi^{\perp}$, $\lambda\in \reals$ and $a\in \Alg$ then
  \begin{enumerate}[label=\emph{($\arabic*$)}]
  \item $D_{\delta_1}(\xi_1+\xi_2)=D_{\delta_1}\xi_1+D_{\delta_1}\xi_1$,
  \item $D_{\lambda\delta_1+\delta_2}\xi_1=\lambda D_{\delta_1}\xi_1+D_{\delta_2}\xi_1$,
  \item $D_{\delta_1}(\xi_1 a)=(D_{\delta_1}\xi_1)a+\xi_1\psi(\delta_1)(a)$.
  \end{enumerate}
\end{proposition}
\begin{proof}
  Note that (1) and (2) follows immediately from the linearity of
  $\nabla$.  To prove (3), one computes the left-hand side directly:
  \begin{align*}
    D_{\delta_1}(\xi_1 a)&=\Pi(\nabla_{\psi(\delta_1)} \xi_1 a)=\Pi((\nabla_{\psi(\delta_1)} \xi_1)a+\xi_1\psi(\delta_1)(a))\\
                         &=\Pi((\nabla_{\psi(\delta_1)} \xi_1)a)+\Pi(\xi_1\psi(\delta_1)(a))=(D_{\delta_1}\xi_1)a+\xi_1\psi(\delta_1)(a),
  \end{align*}
  giving the desired result.
\end{proof}

\noindent
A classical formula in Riemannian geometry is Gauss' equation, which
relates the curvature of the ambient space to the curvature of the
submanifold. The next result provides a noncommutative analogue.

\begin{proposition}[Gauss' equation]
  Let $\delta_i\in\g'$, $\d_i=\psi(\delta_i)\in\g$,
  $E_i=\Psi(\delta_i)\in\MPsi$ and $E_i'=\varphi'(\delta_i)\in M'$ for
  $i=1,2,3,4$ (i.e. $E_i$ is an extension of $E_i'$). Then
  \begin{equation}
    \begin{split}
      \phi\paraa{h(E_1,R(\d_3,\d_4)E_2)}
      =h'&\paraa{E_1',R'(\delta_3,\delta_4)E_2'}
      +\phi\para{h\paraa{\alpha(\delta_4,E_1),\alpha(\delta_3,E_2)}}\\
      &-\phi\para{h\paraa{\alpha(\delta_3,E_1),\alpha(\delta_4,E_2)}}.
    \end{split}
  \end{equation}
\end{proposition}

\begin{proof}
  Using the result from Proposition~\ref{prop:ConnectionExtension} one gets that
  \begin{align*}
    R'(\delta_3,\delta_4)E'_2
    &=\nabla'_3\nabla'_4 E_2'-\nabla'_4\nabla'_3 E_2'-\nabla'_{[\delta_3,\delta_4]} E_2'\\
    &=\nabla'_3\psih(\hat{\nabla}'_4 E_2)-\nabla'_4\psih(\hat{\nabla}'_3 E_2)-\psih(\hat{\nabla}'_{[\delta_3,\delta_4]} E_2)\\
    &=\psih\para{\hat{\nabla}'_3\hat{\nabla}'_4 E_2-\hat{\nabla}'_4\hat{\nabla}'_3 E_2-\hat{\nabla}'_{[\delta_3,\delta_4]} E_2}.
  \end{align*}
  Setting
  $\hat{R}(\der_3,\der_4)E_2:=\hat{\nabla}'_3\hat{\nabla}'_4
  E_2-\hat{\nabla}'_4\hat{\nabla}'_3
  E_2-\hat{\nabla}'_{[\delta_3,\delta_4]} E_2$ one obtains
  \begin{align*}
    h'\paraa{E_1',R'(\delta_3,\delta_4)E_2'}
    &=h'\paraa{\psih(E_1),\psih(\hat{R}(\der_3,\der_4)E_2)}=\phi(h\paraa{E_1,\hat{R}(\der_3,\der_4)E_2})\\
    &=\phi\para{h\paraa{E_1,\hat{\nabla}'_3\hat{\nabla}'_4 E_2-\hat{\nabla}'_4\hat{\nabla}'_3 E_2-\hat{\nabla}'_{[\delta_3,\delta_4]} E_2}}\\
    &=\phi\para{h\paraa{E_1,\nabla_3\hat{\nabla}'_4 E_2-\nabla_4\hat{\nabla}'_3 E_2-\nabla_{[\der_3,\der_4]} E_2}},
  \end{align*}
  since $E_1\in\MPsi$.
  Using the fact that
  $\nabla_i\hat{\nabla}'_j E_k=\nabla_i(\nabla_j
  E_k-\alpha(\delta_j,E_k))$ one may write
  \begin{equation*}
    h'\paraa{E_1',R'(\delta_3,\delta_4)E_2'}=
    \phi\para{h\paraa{E_1,R(\der_3,\der_4)E_2-\nabla_3\alpha(\delta_4,E_2)+\nabla_4\alpha(\delta_3,E_2)}},
  \end{equation*}
  and from this it follows immediately that 
  \begin{equation*}
    \begin{split}
      \phi\paraa{h(E_1,R(\d_3,\d_4)E_2)}
      =h'\paraa{E_1',R'(\delta_3,\delta_4)E_2'}
      &+\phi\para{h\paraa{E_1,\nabla_3\alpha(\delta_4,E_2)}}\\
      &-\phi\para{h\paraa{E_1,\nabla_4\alpha(\delta_3,E_2)}}.
    \end{split}
  \end{equation*}
  Since $(C_{\Alg},h,\nabla)$ is metric it follows that
  \begin{align*}
   h(E_1,\nabla_{\psi(\delta)}\xi)=-h(\nabla_{\psi(\delta)}E_1,\xi) 
  \end{align*}
  for $\xi\in \MPsi^{\perp}$, implying that
  \begin{equation*}
    \begin{split}
      \phi\paraa{h(E_1,R(\d_3,\d_4)E_2)}
      =h'\paraa{E_1',R'(\delta_3,\delta_4)E_2'}
      &+\phi\para{h\paraa{\nabla_4 E_1,\alpha(\delta_3,E_2)}}\\
      &-\phi\para{h\paraa{\nabla_3 E_1,\alpha(\delta_4,E_2)}},
    \end{split}
  \end{equation*}
  which completes the proof, since
  $h\paraa{\nabla_4
    E_1,\alpha(\delta_3,E_2)}=h\paraa{\alpha(\delta_4,E_1),\alpha(\delta_3,E_2)}$
  and
  $h\paraa{\nabla_3
    E_1,\alpha(\delta_4,E_2)}=h\paraa{\alpha(\delta_3,E_1),\alpha(\delta_4,E_2)}$.
\end{proof}

\section{Free real calculi and noncommutative mean curvature}

\noindent
In the examples we shall consider (the noncommutative torus and the
noncommutative 3-sphere), $M$ will be a free module with a basis given
by the image of a basis of the Lie algebra $\g$. Needless to say, the
fact that $M$ is a free module implies several
simplifications. Although it happens for the torus and the 3-sphere
that their modules of vector fields are free (i.e they are
parallelizable manifolds), one expects a projective module in
general. However, as originally shown in the case of the
noncommutative 4-sphere \cite{aw:cgb.sphere}, real calculi can provide
a way of performing local computations, in which case the (localized)
module of vector fields is free.

\begin{definition}
  A real calculus $C_{\Alg}=\RCa$ is called \emph{free} if there
  exists a basis $\der_1,...,\der_m$ of $\g$ such that
  $\varphi(\der_1),...,\varphi(\der_m)$ is a basis of $\Mod$ as a
  (right) $\Alg$-module.
\end{definition}

\noindent
Note that if there exists a basis $\d_1,\ldots,\d_m$ of $\g$ such that
$\varphi(\d_1),\ldots,\varphi(\d_m)$ is a basis of $M$, then
$\varphi(\d_1'),\ldots,\varphi(\d_m')$ is a basis of $M$ for any basis
$\d_1',\ldots,\d_m'$ of $\g$.

\begin{definition}
  A real metric calculus $(C_{\Alg},h)$ is called \emph{free} if
  $C_{\Alg}$ is free and $h$ is invertible.
\end{definition}

\noindent 
An immediate consequence of having an invertible metric, is the
existence of a Levi-Civita connection.

\begin{proposition}
  Let $(C_{\Alg},h)$ be a free real metric calculus. Then there exists
  a unique affine connection $\nabla$ such that $(C_{\Alg},h,\nabla)$
  is a pseudo-Riemannian calculus.
\end{proposition}

\begin{proof}
  Let $\{\der_i\}$ be a basis of $\g$. Since $C_{\Alg}$ is free it
  follows that $E_i=\varphi(\der_i)$ provide a basis of
  $\Mod$. In this basis one gets the components $h_{ij}=h(E_i,E_j)$ of
  the metric $h$, and for notational convenience we set
  $h_{i,[j,k]}:=h(E_i,\varphi[\der_j,\der_k])$ and define $K_{ijk}\in\Alg$ as
  \begin{equation*}
    K_{ijk}:=\frac{1}{2}\para{\der_i h_{jk}+\der_j h_{ik}-\der_k h_{ij}-h_{i,[j,k]}+h_{j,[k,i]}+h_{k,[i,j]}}.
  \end{equation*}
  Now, define the linear functional $\hat{K}_{ij}\in \Mod^*$ by
  \begin{equation*}
    \hat{K}_{ij}(E_k b^k):=K_{ijk}b^k.
  \end{equation*}
  Since the metric $h$ is invertible,
  $m_{ij}=\hat{h}^{-1}(\hat{K}_{ij})\in \Mod$ is well-defined,
  and
  \begin{align*}
    2h(m_{ij},E_k)&=2\hat{h}(m_{ij})(E_k)=2\hat{K}_{ij}(E_k)=2K_{ijk}\\
                  &=\der_i h_{jk}+\der_j h_{ik}-\der_k h_{ij}-h_{i,[j,k]}+h_{j,[k,i]}+h_{k,[i,j]}.
  \end{align*}
  From Corollary~\ref{cor:calculating.Levi.Civita} it now follows that
  there exists a connection $\nabla$ such that $(C_{\Alg},h,\nabla)$
  is pseudo-Riemannian, and from
  Theorem~\ref{thm:Levi.Civita.connection} it follows that $\nabla$ is
  unique.
\end{proof}

\noindent
Given a free real metric calculus $(\CA,h)$ and a basis
$\d_1,\ldots,\d_m$ of $\g$, we write
\begin{align*}
  E_a=\varphi(\d_a)\qquad
  h_{ab}=h(E_a,E_b)\qquad
  [\d_a,\d_b]=f_{ab}^c\d_c
\end{align*}
with $f_{pq}^r\in\reals$, giving
$h(E_a,\varphi([\der_b,\der_c]))=h_{ar}f_{bc}^r$. The fact
that $\hh$ is invertible and $\{E_a\}_{a=1}^m$ is a basis of $M$,
implies that there exists $h^{ab}\in\A$ such that
\begin{align*}
  \hh^{-1}(\hat{E}^a) = E_bh^{ba}\implies
  h^{ab}=\hat{E}^a\paraa{\hh^{-1}(\hat{E}^b)}=h\paraa{\hh^{-1}(\hat{E}^a),\hh^{-1}(\hat{E}^b)}
\end{align*}
where $\{\hat{E}^a\}_{a=1}^m$ is the basis of $M^\ast$ dual to
$\{E_a\}_{a=1}^m$.  It follows that $(h^{ab})^\ast=h^{ba}$ and
\begin{align*}
  h^{ab}h_{bc}=h_{cb}h^{ba} = {\delta^a}_c\mid.
\end{align*}
For a free real metric calculus,we introduce the \emph{Christoffel
  symbols} $\Gamma^a_{bc}\in\A$ as the (unique) coefficients
$\nabla_b E_c=E_a\Gamma^a_{bc}$. Let us now derive an explicit formula
for the Christoffel symbols in terms of the components of the
metric. Indeed, by Koszul's formula it follows that
\begin{equation*}
  h(E_a\Gamma^a_{bc},E_d)=h(\nabla_b E_c,E_d)=
  \frac{1}{2}\left(\der_b h_{cd}+\der_c h_{bd}-\der_d h_{bc}-h_{br}f_{cd}^r+h_{cr}f_{db}^r+h_{dr}f_{bc}^r\right),
\end{equation*}
and since the right hand side is hermitian, one obtains
\begin{equation*}
  h_{da}\Gamma^a_{bc}=\frac{1}{2}\left(\der_b h_{cd}+\der_c h_{bd}-\der_d h_{bc}-h_{br}f_{cd}^r+h_{cr}f_{db}^r+h_{dr}f_{bc}^r\right).
\end{equation*}
Multiplying from the left by $h^{pd}$ gives
\begin{align}
  \Gamma^p_{bc}
  &=\frac{1}{2}h^{pd}\left(\der_b h_{cd}+\der_c h_{bd}-\der_d h_{bc}-h_{br}f_{cd}^r+h_{cr}f_{db}^r\right)+f^p_{bc}\One
\end{align}
and, in particular, if $[\d_a,\d_b]=0$ for all $a,b=1,\ldots,m$ then
\begin{equation}\label{eqn:Christoffel.symbols}
  \Gamma^a_{bc}=\frac{1}{2}h^{ad}\left(\der_b h_{cd}+\der_c h_{bd}-\der_d h_{bc}\right),
\end{equation}
in correspondence with the classical formula.

Let $(\CA,h)$ and $(\CAp,h')$ be free real metric calculi and let
$(\phi,\psi,\psih):(\CA,h)\to(\CAp,h')$ be an isometric embedding.
Since $\psi$ is injective, it is easy to see that if
$\{\delta_i\}_{i=1}^{m'}$ is a basis of $\g'$, then
$\{\Psi(\delta_i)\}_{i=1}^{m'}$ is a basis of $\MPsi$, implying that
$\MPsi$ is a free module of rank $m'$.
Let us now proceed to the define mean curvature, as well as
minimality, of an embedding of free real metric calculi. Since we are
working with extensions of vector fields on the embedded manifold
$\Sigma'$, rather than tangent vectors at points on $\Sigma'$, it is
more natural to consider the restriction (to $\Sigma'$) of the inner
product of the mean curvature vector with an arbitrary vector, rather
than the mean curvature vector itself.

\begin{definition}\label{def:mean.curvature}
  Let $(C_{\Alg},h)$ and $(C_{\Alg'},h')$ be free real metric calculi
  and let $(\phi,\psi,\psih):(\CA,h)\to(\CAp,h')$ be an isometric
  embedding.  Given a basis $\{\delta_i\}_{i=1}^{m'}$ of $\g'$, the
  \emph{mean curvature} $H_{\Algprim}:\Mod\rightarrow \Algprim$ of the
  embedding is defined as
  \begin{equation}\label{eq:def.mean.curvature}
    H_{\Algprim}(m)=\phi\para{h\paraa{m,\alpha(\delta_i,\Psi(\delta_j))}}h'^{ij},
  \end{equation}
  giving trivially $H_{\Algprim}(m)=0$ for $m\in \MPsi$. An embedding
  is called \emph{minimal} if $H_{\Algprim}(\xi)=0$ for all
  $\xi\in \MPsi^{\perp}$.
\end{definition}

\begin{remark}
  Note that the ordering in \eqref{eq:def.mean.curvature} is natural
  in the following sense. Considering the restriction of the metric
  $h$ to $\MPsi$, given by $h_{ij}=h(\Psi(\delta_i)),\Psi(\delta_j))$
  and its inverse $h^{ij}$, the fact that $M$ is a right module gives
  a natural definition of the mean curvature as
  \begin{align*}
    H_{\Algprim}(m)
    &=\phi\para{h\paraa{m,\alpha(\delta_i,\Psi(\delta_j))h^{ij}}}
    =\phi\para{h\paraa{m,\alpha(\delta_i,\Psi(\delta_j))}}\phi(h^{ij})\\
    &=\phi\para{h\paraa{m,\alpha(\delta_i,\Psi(\delta_j))}}h'^{ij},
  \end{align*}
  reproducing the formula in Definition~\ref{def:mean.curvature}.
\end{remark}

\noindent
Although defined with respect to a basis of $\g'$, the mean curvature
is independent of the choice of basis. Indeed, if we let $h'_{ij}$
and $\tilde{h}'_{ij}$ denote the components of the metric $h'$ with
respect to different bases $\{\delta_i\}$ and $\{\tilde{\delta}_i\}$
of $\g'$, then there exists a (real) invertible matrix $A$ such that
$\tilde{h'}=Ah'A^T$, or equivalently
$\tilde{h'}_{ij}={A^k}_ih'_{kl}{A^{l}}_j$.  Consequently,
\begin{align*}
  (\tilde{h'})^{ij}
  ={(A^{-1})^i}_kh'^{kl}{(A^{-1})^j}_l
\end{align*}
and it follows that the mean curvature calculated using the basis
$\{\tilde{\delta}_i\}$ is
\begin{align*}
  H_{\Alg'}(m)
  &=\phi\para{h\paraa{m,\alpha(\tilde{\delta}_i,\Psi(\tilde{\delta}_j))}}(\tilde{h'})^{ij}\\
  &=\phi\para{h\paraa{m,\alpha({A^k}_i\delta_k,\Psi({A^l}_j\delta_l))}}
    {(A^{-1})^i}_mh'^{mn}{(A^{-1})^j}_n\\
  &={A^k}_i{(A^{-1})^i}_m{A^l}_j{(A^{-1})^j}_n\phi\para{h\paraa{m,\alpha(\delta_k,\Psi(\delta_l))}}h'^{mn}\\
  &=\phi\para{h\paraa{m,\alpha(\delta_k,\Psi(\delta_l))}}(h')^{kl}
\end{align*}
showing that the definition of $H_{\A'}$ is indeed basis independent.

Let us end this section by noting that it is straight-forward to
define the gradient, divergence and Laplace operator for free real
metric calculi.

\begin{definition}
  Let $(C_{\Alg},h)$ be a free real metric calculus and let $\nabla$
  denote the Levi-Civita connection. Moreover, let $\{\d_a\}_{a=1}^m$ be a basis
  of $\g$ and set $E_a=\varphi(\d_a)$.  The \textit{gradient}
  $\operatorname{grad}:\A\to M$ is defined as
  \begin{equation*}
    \operatorname{grad}(a)=E_ah^{ab}\d_b a
  \end{equation*}
  for $a\in\A$. The \emph{divergence} $\operatorname{div}:M\to\A$ is
  defined as
  \begin{align*}
    \operatorname{div}(m) = (\nabla_{\d_a}m)^a
  \end{align*}
  for $m\in M$, where $\nabla_{\d_a}m=E_b(\nabla_{\d_a}m)^b$. The
  Laplace operator $\Delta:\A\to\A$ is defined as
  $\Delta(a)=\operatorname{div}\paraa{\operatorname{grad}(a)}$ for
  $a\in\A$.
\end{definition}

\noindent
Note that it is easy to check that the above definitions are
independent of the choice of basis of $\g$.

\section{Minimal tori in the 3-sphere}\label{sec:torus.embedded.into.3sphere}

\noindent
The 3-sphere has a rich flora of minimal surfaces, and the fact that
minimal surfaces of arbitrary genus exist in $S^3$ is a famous result
by Lawson \cite{l:completeminimalsurfaces}. As an illustration of the
concepts we have developed, as well as being our motivating example,
we shall consider the noncommutative torus minimally embedded in the
noncommutative 3-sphere. However, rather than the round metric on
$S^3$, we will consider more general metrics. Therefore, let
us start by recalling the classical situation.

The Clifford torus $T^2$ is embedded in $S^3\subseteq \mathbb{R}^4$ via
\begin{equation*}
\vec{x}=\frac{1}{\sqrt{2}}(x^1,x^2,x^3,x^4)=(\cos \varphi_1,\sin \varphi_1,\cos \varphi_2,\sin \varphi_2).
\end{equation*}
With $\delta_1=\d_{\varphi_1}$ and $\delta_2=\d_{\varphi_2}$, the tangent
space at a point is spanned by
\begin{align*}
\delta_1 \vec{x}&=\frac{1}{\sqrt{2}}(-\sin \varphi_1, \cos \varphi_1,0,0)=(-x^2,x^1,0,0)\\
\delta_2 \vec{x}&=\frac{1}{\sqrt{2}}(0,0,-\sin \varphi_2,\cos \varphi_2)=(0,0,-x^4,x^3).
\end{align*}
The 3-sphere is embedded in $\complex^2$ via
\begin{align*}
  z&=e^{i\xi_1}\sin \eta\\ w&=e^{i\xi_2}\cos \eta,
\end{align*}
and with $\d_1=\d_{\xi_1}$ and $\d_2=\d_{\xi_2}$ the tangent space at
a point with $0 < \xi_1,\xi_2 < 2\pi$ and $0 < \eta < \pi/2$ is
spanned by
\begin{align*} E_1&=\der_1(x^1,x^2,x^3,x^4)=(-x^2,x^1,0,0)\\
E_2&=\der_2(x^1,x^2,x^3,x^4)=(0,0,-x^4,x^3)\\
E_{\eta}&=\der_{\eta}(x^1,x^2,x^3,x^4)=(\cos\xi_1\cos\eta,\sin\xi_1\cos\eta,-\cos\xi_2\sin\eta,-\sin\xi_2\sin\eta).
\end{align*} 
\noindent
The standard metric on $S^3$ is given by
\begin{equation*}
  g=\begin{pmatrix}
    \sin^2{\eta} & 0 & 0 \\
    0 & \cos^2{\eta} & 0 \\
    0 & 0 & 1 
  \end{pmatrix},
\end{equation*}
and for $H\in C^{\infty}(S^3)$ such that $H>0$ we consider the perturbed metric
\begin{equation*}
  \tilde{g}=H\begin{pmatrix}
    \sin^2{\eta} & 0 & 0 \\
    0 & \cos^2{\eta} & 0 \\
    0 & 0 & 1 
  \end{pmatrix}H.
\end{equation*}
Let us now proceed to determine the Levi-Civita connection on
$(S^3,\tilde{g})$.  The Christoffel symbols are computed using
\begin{equation*}
  \Gamma^i_{jk}=\frac{1}{2}\tilde{g}^{il}\left(\der_j\tilde{g}_{kl}+\der_k\tilde{g}_{jl}-\der_l\tilde{g}_{jk}\right),
\end{equation*}
giving
\begin{equation*}
  \Gamma^1_{jk}=\begin{pmatrix}
    \der_1 (\ln H) & \der_2 (\ln H) & \der_{\eta} (\ln H)+\cot{\eta}\\
    \der_2 (\ln H) &-\der_1(\ln H)\cot^2{\eta}& 0\\
    \der_{\eta} (\ln H)+\cot{\eta} & 0 & -\der_1(\ln H)\csc^2{\eta}
  \end{pmatrix}
\end{equation*}
\begin{equation*}
  \Gamma^2_{jk}=\begin{pmatrix}
    -\der_2 (\ln H)\tan^2{\eta} & \der_1 (\ln H) & 0\\
    \der_1 (\ln H) & \der_2(\ln H) & \der_{\eta}(\ln H)-\tan{\eta}\\
    0 & \der_{\eta}(\ln H)-\tan{\eta} & -\der_2(\ln H)\sec^2{\eta}
  \end{pmatrix}
\end{equation*}
\begin{equation*}
  \Gamma^3_{jk}=\begin{pmatrix}
    -\der_{\eta} (\ln H)\sin^2{\eta}-\sin{\eta}\cos{\eta} & 0 & \der_1(\ln H)\\
    0 & -\der_{\eta}(\ln H)\cos^2{\eta}+\sin{\eta}\cos{\eta} & \der_2(\ln H)\\
    \der_1(\ln H) & \der_2(\ln H) & \der_{\eta}(\ln H)
  \end{pmatrix}.
\end{equation*}
Thus, the Levi-Civita connection is explicitly given as

\begin{align*}
  &\nabla_{1}\der_1=\der_1(\ln H)\der_1-\der_2(\ln H)\tan^2{\eta}\der_2-(\der_{\eta}(\ln H)\sin^2{\eta}+\sin{\eta}\cos{\eta})\d_{\eta}\\
  &\nabla_{1}\der_2=\der_2(\ln H)\der_1+\der_1(\ln H)\der_2=\nabla_{2}\der_1 \\
  &\nabla_{1}\der_{\eta}=(\der_{\eta}(\ln H)+\cot{\eta})\der_1+\der_1(\ln H)\d_{\eta}=\nabla_{\eta}\der_1 \\
  &\nabla_{2}\der_2=-\der_1(\ln H)\cot^2{\eta}\der_1+\der_2(\ln H)\der_2+(\sin{\eta}\cos{\eta}-\der_{\eta}(\ln H)\cos^2{\eta})\d_{\eta} \\
  &\nabla_{2}\der_{\eta}=(\der_{\eta}(\ln H)-\tan{\eta})\der_2+\der_2(\ln H)\d_{\eta}=\nabla_{\eta}\der_2 \\
  &\nabla_{\eta}\der_{\eta}=-\der_1(\ln H)\csc^2{\eta}\der_1-\der_2(\ln H)\sec^2{\eta}\der_2+\der_{\eta}(\ln H)\d_{\eta}.
\end{align*}

\subsection{Embedding the torus into the 3-sphere}

For fixed $\eta_0\in(0,\pi/2)$, let $f_{\eta_0}:T^2\rightarrow (S^3,\tilde{g})$ denote the
embedding
\begin{align*}
  f_{\eta_0}:(\cos \varphi_1,\sin \varphi_1,\cos \varphi_2,\sin \varphi_2)\mapsto
  (e^{i\varphi_1}\sin\eta_0,e^{i\varphi_2}\cos\eta_0),
\end{align*}
The induced metric on the torus is given by
\begin{equation*}
  g_{T^2}=\tilde{H}
  \begin{pmatrix}
    \sin^2{\eta_0} & 0 \\
    0 & \cos^2{\eta_0}
  \end{pmatrix}\tilde{H},
\end{equation*}
where $\tilde{H}(\varphi_1,\varphi_2)=H(\varphi_1,\varphi_2,\eta_0)$.
The unit normal of $T^2$ is $N=\tilde{H}^{-1}\d_{\eta}$, and one
writes the second fundamental form $\alpha$ as:
\begin{align*}
  \alpha(\delta_1,\delta_1)&=-\tilde{H}(\d_{\eta}(\ln H)|_{\eta_0}\sin^2{\eta_0}+\sin{\eta_0}\cos{\eta_0})N\\
  \alpha(\delta_1,\delta_2)&=\alpha(\delta_2,\delta_1)=0\\
  \alpha(\delta_2,\delta_2)&=\tilde{H}(\sin{\eta_0}\cos{\eta_0}-\der_{\eta}(\ln H)|_{\eta_0}\cos^2{\eta_0})N.
\end{align*}
Calculating the mean curvature of $T^2$ in $(S^3,\tilde{g})$ yields
\begin{align*}
  H_{T^2}&=\frac{1}{2}\tilde{g}\paraa{N,\alpha\paraa{\delta_i,\delta_j}}g_{T^2}^{ij}
  =-\tilde{H}^{-1}(\cot{2\eta_0}+\der_{\eta}(\ln H)|_{\eta_0})
\end{align*}
and it follows that $T^2$ is minimally embedded in $(S^3,\tilde{g})$
if $\d_{\eta}(\ln H)|_{\eta_0}=-\cot 2\eta_0$; for instance, one might
choose 
\begin{equation*}
  H(\xi_1,\xi_2,\eta)=\exp\left(p(\xi_1,\xi_2)-\frac{r(\eta)\cot{2\eta_0}}{r'(\eta_0)}\right)
\end{equation*}
for arbitrary functions $p$ and $r$, with $r$ having a nonzero derivative at
$\eta=\eta_0$. In the classical case, when $H=1$, the embedding is
minimal if $\cot 2\eta_0=0$, i.e. $\eta_0=\pi/4$.

\section{The noncommutative minimal torus}\label{sec:NCtorus.NC3sphere}

\noindent
Let us now apply the framework for noncommutative embeddings to the
case of the noncommutative torus and the noncommutative 3-sphere. We
shall start by recalling their definitions, as well as their
corresponding real metric calculi. For more details, we refer to
\cite{aw:curvature.three.sphere} (however, where only the standard
metric on the 3-sphere was considered).

\subsection{The noncommutative torus}

The noncommutative torus $\Torus$ is a unital $\ast$-algebra generated
by the unitary elements $U,V$ subject to the relation $VU=qUV$, with
$q=e^{2\pi i\theta}$.  Introducing the hermitian elements
\begin{alignat*}{2}
  X^1&=\frac{1}{2}(U+U^*) &\qquad X^2&=\frac{1}{2i}(U-U^*)\\
  X^3&=\frac{1}{2}(V+V^*) & X^4&=\frac{1}{2i}(V-V^*)
\end{alignat*}
gives $\One=UU^*=(X^1)^2+(X^2)^2$ and $\One=VV^*=(X^3)^2+(X^4)^2$.  In
analogy with the geometrical setting, let $\Modprim$ be the (right)
submodule of $(\Torus)^4$ generated by
\begin{align*}
  e_1&=(-X^2,X^1,0,0)\\
  e_2&=(0,0,-X^4,X^3).
\end{align*}
We note that $\Modprim$ is a free $\Torus$-module, since $e_1$ and $e_2$ form a basis for $\Modprim$:
\begin{align*}
  e_1a+e_2b=0 &\implies (-X^2a,X^1a,-X^4b,X^3b)=(0,0,0,0)\\
              &\implies 
                \left\{\begin{array}{lr}
                         \FofG{}{(X^1)^2+(X^2)^2}a=UU^*a=a=0\\
                         \FofG{}{(X^3)^2+(X^4)^2}b=VV^*b=b=0.
                       \end{array}\right.
\end{align*}
Next, we let $\Derprim$ be  the (real) Lie algebra generated by the two hermitian derivations $\delta_1,\delta_2$, given by
\begin{alignat*}{2}
  \delta_1 U&=iU &\qquad \delta_1 V&=0\\
  \delta_2 U&=0  & \delta_2 V&=iV,
\end{alignat*}
satisfying $[\delta_1,\delta_2]=0$. Finally, let
$\varphi':\Derprim\rightarrow \Modprim$ with $\varphi'(\delta_j)=e_j$
for $j=1,2$ and extended by $\Real$-linearity, which implies that
$\Modprim$ is generated by $\varphi'(\Derprim)$ as a
$\Torus$-module. Hence, we have shown that
$C_{\Torus}=(\Torus,\Derprim,\Modprim,\varphi')$ is a real calculus
over the noncommutative torus.

As a first illustration of a real calculus homomorphism, let us
construct a family of automorphisms of $\Torus$ as
follows. Let $a,b,c,d\in \mathbb{Z}$ be given such that $ad-bc=1$,
and let $\alpha:\Torus\rightarrow\Torus$ be the automorphism given by
\begin{align*}
	\alpha(U)&=U^aV^b,\\
    \alpha(V)&=U^cV^d,
\end{align*}
with inverse
\begin{align*}
  \alpha^{-1}(U)&=q^{\frac{1}{2}bd(a-c-1)}U^dV^{-b}\\
  \alpha^{-1}(V)&=q^{\frac{1}{2}ac(d-b-1)}U^{-c}V^{a}.                  
\end{align*}
Once the automorphism $\alpha$ is established, it is a simple task to
find a real calculus automorphism from $C_{\Torus}$ to itself by using
Proposition~\ref{prop:RCIsoLemma} to find the required Lie algebra
homomorphism.  Indeed, Proposition~\ref{prop:RCIsoLemma} implies that
\begin{align*}
  \psi(\delta_1)(U)&=\alpha^{-1}\circ\delta_1\circ\alpha(U)=iaU & \psi(\delta_2)(U)&=\alpha^{-1}\circ\delta_2\circ\alpha(U)=ibU\\
    \psi(\delta_1)(V)&=\alpha^{-1}\circ\delta_1\circ\alpha(V)=icV  & \psi(\delta_2)(V)&=\alpha^{-1}\circ\delta_2\circ\alpha(V)=idV,
\end{align*}
giving
\begin{align*}
  \psi(\delta_1)&=a\delta_1+c\delta_2\mathand
  \psi(\delta_2)=b\delta_1+d\delta_2.
\end{align*}
From the compatibility conditions $\psih(\Psi(\delta_i))=\varphi'(\delta_i)$ one obtains
\begin{align*}
  \psih\paraa{e_1a+e_2c} = e_1\mathand
  \psih\paraa{e_1b+e_2d}= e_2,
\end{align*}
implying that
\begin{align*}
  \psih(e_1) = e_1d-e_2c\mathand
  \psih(e_2) = -e_1b+e_2a.
\end{align*}
This ensures that $(\alpha,\psi,\psih)$ as defined above is an automorphism of the real calculus $C_{\Torus}$.

\subsection{The noncommutative 3-sphere}
The noncommutative 3-sphere $\Sphere$ is the unital $\ast$-algebra
generated by $Z,Z^*,W,W^*$ satisfying
\begin{align*}
  WZ&=qZW &W^*Z&=\bar{q}ZW^* &WZ^*&=\bar{q}Z^*W
  &W^*Z^*=qZ^*W^*\\ Z^*Z&=ZZ^* &W^*W&=WW^* &WW^*&=\One-ZZ^*,&
\end{align*}
with $q=e^{2\pi i\theta}$ for $\theta\in\reals$.

Similar to the case of $\Torus$, we introduce
\begin{align*} X^1&=\frac{1}{2}(Z+Z^*) & X^2&=\frac{1}{2i}(Z-Z^*)\\
X^3&=\frac{1}{2}(W+W^*) & X^4&=\frac{1}{2i}(W-W^*)\\ |Z|^2&=ZZ^* &
|W|^2&=WW^*,
\end{align*}
giving $|Z|^2=(X^1)^2+(X^2)^2$ and $|W|^2=(X^3)^2+(X^4)^2$; recall
that $|Z|^2$ and $|W|^2$ are in the center of $\Sphere$ and,
furthermore, neither of them is a zero divisor
(cf. \cite{aw:curvature.three.sphere}). Let us now construct a real
metric calculus for $\Sphere$, closely related to the Hopf fibration
of the 3-sphere.

Recall from Section~\ref{sec:torus.embedded.into.3sphere} that $S^3$
can be given in terms of the coordinates $(\xi_1,\xi_2,\eta)$, and we
noted that the tangent plane at a given point is spanned by the three
vectors
\begin{align*} E_1&=\der_1(x^1,x^2,x^3,x^4)=(-x^2,x^1,0,0)\\
E_2&=\der_2(x^1,x^2,x^3,x^4)=(0,0,-x^4,x^3)\\
E_{\eta}&=\der_{\eta}(x^1,x^2,x^3,x^4)=(\cos\xi_1\cos\eta,\sin\xi_1\cos\eta,-\cos\xi_2\sin\eta,-\sin\xi_2\sin\eta).
\end{align*}
For the noncommutative analogue, it is apparent how to choose $E_1$
and $E_2$, but the analogue of $E_{\eta}$ is less clear. Therefore,
instead of $\der_{\eta}$, one considers the derivation
$\der_3=|z||w|\der_{\eta}$, giving
\begin{equation*}
E_3=\der_3(x^1,x^2,x^3,x^4)=(x^1|w|^2,x^2|w|^2,-x^3|z|^2,-x^4|z|^2),
\end{equation*} which can be used together with $E_1$ and $E_2$ to
span the tangent space.

Returning to the complex embedding coordinates $z$ and $w$ in
$\mathbb{C}^2$, one finds
\begin{align}
  \der_1(z)&=iz &\der_1(w)&=0 \\
  \der_2(z)&=0
                &\der_2(w)&=iw \\
  \der_3(z)&=z|w|^2 &\der_3(w)&=-w|z|^2,
\end{align}
and with respect to the basis $\{E_1,E_2,E_3\}$ of the tangent space
of $S^3$, the induced standard metric becomes
\begin{equation} (h_{ab})=(h(E_a,E_b))=\begin{pmatrix} |z|^2 & 0 & 0
\\ 0 & |w|^2 & 0 \\ 0 & 0 & |z|^2|w|^2
\end{pmatrix}.
\end{equation}
Motivated by the above considerations, let $M$ the submodule of the free
(right) module $(\Sphere)^4$ generated by $\{E_1,E_2,E_3\}$, where
\begin{align*} E_1&=(-X^2,X^1,0,0)\\ E_2&=(0,0,-X^4,X^3)\\
E_3&=(X^1|W|^2,X^2|W|^2,-X^3|Z|^2,-X^4|Z|^2).
\end{align*}
In \cite{aw:curvature.three.sphere} it was shown that $M$ is a free
module with a basis $\{E_1,E_2,E_3\}$ and that there exist hermitian
derivations $\d_1,\d_2,\d_3$ such that
\begin{align*}
  \der_1(Z)&=iZ & \der_1(W)&=0\\ \der_2(Z)&=0 &
                                                \der_2(W)&=iW\\ \der_3(Z)&=Z|W|^2 & \der_3(W)&=-W|Z|^2,
\end{align*}
with $[\der_a,\der_b]=0$ for $a,b=1,2,3$.  Let $\Der$ be the (real)
Lie algebra generated by $\der_1,\der_2$ and $\der_3$, and define
$\varphi:\Der\rightarrow\Mod$ as the linear map (over $\mathbb{R}$)
given by $\varphi(\der_a)=E_a$ for $a=1,2,3$.  From the above
considerations, it follows that $C_{\Sphere}=(\Sphere,\g,M,\varphi)$ is
a real calculus over $\Sphere$.

Now, let us proceed to construct a real metric calculus over
$\Sphere$, in which we shall minimally embed the noncommutative torus.
In analogy with Section~\ref{sec:torus.embedded.into.3sphere}, we
choose the hermitian form $h:M\times M\to M$
\begin{equation*}
  h(m,n)=\sum_{a,b=1}^3 (m^a)^*h_{ab}n^b,
\end{equation*}
where $m=E_a m^a,n=E_b n^b$ and
\begin{equation*}
  (h_{ab})=H\begin{pmatrix} |Z|^2 & 0 & 0 \\ 0 &
    |W|^2 & 0 \\ 0 & 0 & |Z|^2|W|^2 \\
\end{pmatrix}H^*,
\end{equation*}
where $H\in\Sphere$ is chosen such that $HH^*$ is invertible. Since neither $|Z|^2$ nor $|W|^2$ is a zero divisor, the
metric is clearly non-degenerate; furthermore, $h_{ab}$ is hermitian
for $a,b=1,2,3$. We conclude that $(C_{\Sphere},h)$ is a real metric
calculus.

Next, let us construct a metric and torsion-free connection on
$(C_{\Sphere},h)$. In order achieve this, we will localize the
algebra at $|Z|^2$ and $|W|^2$. That is, one extends the algebra of
the noncommutative 3-sphere by the inverses of $|Z|^2$ and
$|W|^2$. (In principle, for a well-behaved noncommutative
localization, one has to check the so called Ore conditions, but since
$|Z|^2$ and $|W|^2$ are central, these are trivially fulfilled.) The
resulting algebra is denoted by $\Sthreetloc$. It is straight-forward
to extend the real metric calculus $(C_{\Sphere},h)$ to a real metric
calculus $(C_{\Sthreetloc},h)$ (cf. \cite{aw:cgb.sphere} where a
similar construction was carried out for the 4-sphere).

\begin{proposition}
  There exists a unique affine connection $\nabla$ such that
  $(C_{S^3_{\theta,\text{loc}}},h,\nabla)$ is a pseudo-Riemannian
  calculus with
  \begin{align*} \nabla_1 E_1&=E_1 H_1-E_2
                               |Z|^2|W|^{-2}H_2-E_3(|W|^{-2}H_3+\One)\\ \nabla_1
    E_2&=\nabla_2 E_1=E_1 H_2+E_2 H_1\\ \nabla_1 E_3&=\nabla_3
                                                      E_1=E_1 (H_3+|W|^2)+E_3 H_1\\ \nabla_2
    E_2&=-E_1|W|^2|Z|^{-2}H_1+E_2H_2+E_3(\One-|Z|^{-2}H_3)\\
    \nabla_2 E_3&=\nabla_3 E_2=E_2(H_3-|Z|^2)+E_3H_2\\
    \nabla_3
    E_3&=-E_1|W|^2H_1-E_2|Z|^2H_2+E_3(H_3+|W|^2-|Z|^2),
  \end{align*}
  where $H_a=\frac{1}{2}(HH^*)^{-1}\der_a(HH^*)$ for $a=1,2,3$.
\end{proposition}

\begin{proof}
  Since $h$ is invertible, $(C_{\Sthreetloc},h)$ is a free
  real metric calculus, implying that the Levi-Civita connection
  $\nabla$ exists. Moreover, $[\der_a,\der_b]=0$ for all
  $a,b\in \{1,2,3\}$, and thus it follows that the Christoffel symbols
  for $\nabla$ can be calculated directly using
  \eqref{eqn:Christoffel.symbols}.  For instance,
  \begin{align*}
  	&\Gamma^1_{11}=\frac{1}{2}h^{11}\der_1h_{11}=\frac{1}{2}(HH^*)^{-1}\der_1(HH^*)=H_1\\
  	&\Gamma^2_{11}=\frac{1}{2}h^{22}(-\der_2h_{11})=-\frac{1}{2}|Z|^2|W|^{-2}(HH^*)^{-1}\der_2(HH^*)=-|Z|^2|W|^{-2}H_2\\
  	&\Gamma^3_{11}=\frac{1}{2}h^{33}(-\der_3h_{11})=-\frac{1}{2}|W|^{-2}(HH^*)^{-1}\der_3(HH^*)-\One=-|W|^{-2}H_3-\One,
  \end{align*}
  giving
  \begin{align*}
    \nabla_1E_1=E_1H_1-E_2|Z|^2|W|^{-2}H_2-E_3(|W|^{-2}H_3+\One).
  \end{align*}
  The remaining Christoffel symbols are computed in a completely analogous way.
\end{proof}

\subsection{An embedding of the noncommutative torus}

\noindent
Finally, we will now construct an embedding
$(\phi,\psi,\psih):C_{\Sthreetloc}\to C_{\Torus}$. To this end, we set
\begin{align*}
	\phi(Z)&=\lambda U \\
	\phi(W)&=\mu V,
\end{align*}
where $\lambda$ and $\mu$ are complex nonzero constants such that
$|\lambda|^2+|\mu|^2=1$. It is easy to verify that with these
conditions $\phi$ is a $\ast$-algebra homomorphism. Moreover, since
$\lambda$ and $\mu$ are chosen to be nonzero it means that $\phi$ is
surjective as well.  With this choice of $\phi$ it follows that a Lie
algebra homomorphism $\psi:\Derprim\rightarrow\Der$ compatible with
$\phi$ is given by
\begin{align*}
	\psi(\delta_1)=\der_1\mathand
	\psi(\delta_2)=\der_2,
\end{align*}
and $\Mod_{\Psi}$ is the submodule of $\Mod$ generated by $E_1$ and
$E_2$. Furthermore,  with 
\begin{align*}
	\psih(E_1)=e_1\mathand
	\psih(E_2)=e_2
\end{align*}
$(\phi,\psi,\psih)$ is a real calculus homomorphism. This choice of
$(\phi,\psi,\psih)$ gives an embedding of $C_{\Torus}$ into
$C_{\Sthreetloc}$, since by choosing $\tilde{M}$ to be the submodule
of $\Mod$ generated by $E_3$ one gets that
$\Mod=\MPsi\oplus\tilde{M}$.

Let us now find the induced metric $h'$ such that
$(\phi,\psi,\psih):(C_{\Torus},h')\to(C_{\Sthreetloc},h)$ is an
embedding of real metric calculi.  Since $\Modprim$ has a basis
$\{e_1,e_2\}$ it suffices to calculate $h'(e_i,e_j)$ for $i,j=1,2$:
\begin{align*}
h'(e_1,e_1)&=\phi(h(E_1,E_1))=\phi((HH^*)|Z|^2)=|\lambda|^2(\tilde{H}\tilde{H}^*) \\
h'(e_1,e_2)&=h'(e_2,e_1)=\phi(h(E_1,E_2))=0 \\
h'(e_2,e_2)&=\phi(h(E_2,E_2))=\phi((HH^*)|W|^2)=|\mu|^2(\tilde{H}\tilde{H}^*),
\end{align*}
with $\tilde{H}=\phi(H)$; it is easy to check that $h'$ is an
invertible metric on $\Modprim$, implying that $(C_{\Torus},h')$ is
indeed a free real metric calculus. Moreover, it is clear that
$\tilde{M}=\MPsi^\perp$.

Since $\Mod$ and $\Modprim$ are free modules,
Proposition~\ref{prop:ConnectionExtension} can be used to quickly
determine the Levi-Civita connection $\nabla'$ for $(C_{\Torus},h')$:
\begin{align*}
\nabla'_1 e_1&=\psih(L(\delta_1,\Psi(\delta_1)))=e_1\tilde{H}_1-e_2\tilde{H}_2|\lambda|^2|\mu|^{-2} \\
\nabla'_1 e_2&=\nabla'_2 e_1=\psih(L(\delta_1,\Psi(\delta_2)))=e_1\tilde{H}_2+e_2\tilde{H}_1 \\
\nabla'_2 e_2&=\psih(L(\delta_2,\Psi(\delta_2)))=-e_1\tilde{H}_1|\lambda|^{-2}|\mu|^{2}+e_2\tilde{H}_2,
\end{align*}
where $\tilde{H}_i=\phi(H_i)$ for $i=1,2,3$.  Consequently, one
obtains the second fundamental form as
\begin{align*}
\alpha(\delta_1,\Psi(\delta_1))&=-E_3(|W|^{-2}H_3+\One)\\
\alpha(\delta_1,\Psi(\delta_2))&=\alpha(\delta_2,\Psi(\delta_1))=0\\
\alpha(\delta_2,\Psi(\delta_2))&=E_3(\One-|Z|^{-2}H_3),
\end{align*}
giving the mean curvature
\begin{align*}
  H_{T^2_{\theta}}(m)={}&\phi\para{h\paraa{m,\alpha(\delta_1,\Psi(\delta_1))}}(h')^{11}
                          +\phi\para{h\paraa{m,\alpha(\delta_2,\Psi(\delta_2))}}(h')^{22}\\
  ={}&\phi\para{h\paraa{m,-E_3(|W|^{-2}H_3+\One)}}|\lambda|^{-2}(\tilde{H}\tilde{H}^*)^{-1}\\
                        &+\phi\para{h\paraa{m,E_3(\One-|Z|^{-2}H_3)}}|\mu|^{-2}(\tilde{H}\tilde{H}^*)^{-1}\\
  ={}&\phi\left(h\paraa{m,E_3}\right)\left(|\mu|^{-2}-|\lambda|^{-2}-2|\lambda|^{-2}|\mu|^{-2}\tilde{H}_3\right)(\tilde{H}\tilde{H}^*)^{-1}.
\end{align*}
For the embedded torus, $\MPsi^{\perp}$ is the submodule of $\Mod$
generated by the basis element $E_3$. Hence, the mean curvature is
zero if
\begin{align*}
0=H_{T^2_{\theta}}(E_3)&=(\tilde{H}\tilde{H}^*)|\lambda|^2|\mu|^2\left(|\mu|^{-2}-|\lambda|^{-2}-2|\lambda|^{-2}|\mu|^{-2}\tilde{H}_3\right)(\tilde{H}\tilde{H}^*)^{-1}\\
&=(\tilde{H}\tilde{H}^*)\left(|\lambda|^{2}-|\mu|^{2}-2\tilde{H}_3\right)(\tilde{H}\tilde{H}^*)^{-1}\\
&=\left(|\lambda|^2-|\mu|^2\right)\One-2(\tilde{H}\tilde{H}^*)\tilde{H}_3(\tilde{H}\tilde{H}^*)^{-1}\\
&=\left(|\lambda|^2-|\mu|^2\right)\One-\phi(\der_3(HH^*))(\tilde{H}\tilde{H}^*)^{-1},
\end{align*}
implying that the embedding of $(C_{\Torus},h')$ into
$(C_{\Sphere},h)$ is minimal if and only if
\begin{align*}
  \phi\paraa{\der_3(HH^\ast)}=(|\lambda|^2-|\mu|^2)\phi(HH^\ast).
\end{align*}
In the special case where $\phi(\der_3(HH^\ast))=0$, the embedding is
minimal if $|\lambda|=|\mu|=1/\sqrt{2}$ (in analogy with the classical
case). For the same values of $|\lambda|$ and $|\mu|$ one may also
choose, e.g.,  $H=ZW$ giving $HH^\ast=|Z|^2|W|^2$ and 
\begin{align*}
  \phi\paraa{\d_3HH^\ast} = 2|\lambda|^2|\mu|^2\paraa{|\mu|^2-|\lambda|^2} = 0.
\end{align*}

\section*{Acknowledgments}

\noindent
We would like to thank J. Choe for discussions. Furthermore, J.A. is
supported by the Swedish Research Council grant 2017-03710.

\bibliographystyle{alpha}
\bibliography{references}

\end{document}